\documentclass[12pt]{article}
\usepackage{latexsym}
\usepackage{amssymb}
\usepackage{amsmath}
\usepackage{amsthm}
\usepackage{mathrsfs}
\usepackage{wasysym}
\usepackage{enumerate}
\usepackage{rawfonts}
\input{prepictex}
\input{pictex}
\input{postpictex}
\numberwithin{equation}{section}
\usepackage[OT2,OT1]{fontenc}
\def\cyr{%
\renewcommand\rmdefault{wncyr}%
\renewcommand\sfdefault{wncyss}%
\renewcommand\encodingdefault{OT2}%
\normalfont
\selectfont}
\DeclareMathAlphabet{\zap}{OT1}{pzc}{m}{it}
\DeclareTextFontCommand{\textcyr}{\cyr}
\def\be{\begin{equation}}
\def\ee{\end{equation}}
\def\bea{\begin{eqnarray*}}
\def\eea{\end{eqnarray*}}

\newcommand{\rad}{\text{\cyr   ya}}
\newcommand{\eg}{\text{\cyr   \bf G}}
\def\Bbb{\mathbb}

\def\CC{\mathbb{C}}
\def\PP{\mathbb{P}}

\newtheorem{main}{Theorem}

\DeclareMathOperator{\Hom}{Hom}

\newtheorem{thm}{Theorem}[section]
\newtheorem{lem}[thm]{Lemma}
\newtheorem{prop}[thm]{Proposition}
\newtheorem{cor}[thm]{Corollary}
\newtheorem{defn}[thm]{Definition}

\newenvironment{rmk}{\mbox{ }\\{\bf  Remark}\mbox{ }}{
\hfill  $\diamondsuit$\mbox{}\bigskip}
\def\ZZ{{\mathbb{Z}}}
\def\RR{{\mathbb{R}}}
\def\CP{{\mathbb{C} \mathbb{P}}}
\begin{document}

\title{Mass in  K\"ahler Geometry}

\author{Hans-Joachim Hein\thanks{Research funded in part by NSF grant DMS-1514709.} 
~and Claude LeBrun\thanks{Research funded in part by NSF grant DMS-1510094.}}

\date{}
\maketitle

\begin{abstract}
We  prove a simple, 
 explicit formula for the mass of any asymptotically locally Euclidean (ALE) K\"ahler manifold, assuming only 
the sort of weak  fall-off conditions required for the mass to actually be well-defined. 
For ALE  scalar-flat K\"ahler manifolds,
the mass turns out to be a  topological invariant,  depending only
on the  underlying smooth manifold,  the first Chern class of the complex structure, and 
the K\"ahler class of the metric. 
When the metric is actually AE (asymptotically Euclidean), our formula not only 
implies a  positive mass theorem  for K\"ahler metrics, but also yields
a Penrose-type inequality for the mass.
\end{abstract}

A complete connected non-compact  
Riemannian manifold $(M,g)$ of dimension $n\geq 3$  is said to be {\em asymptotically Euclidean} (or {\em AE}\,) if there is 
a compact subset $\mathbf{K}\subset M$ such that $M-\mathbf{K}$ consists of finitely many components, each of which 
is diffeomorphic to the complement of a closed ball $\mathbf{D}^n \subset \RR^n$,  in  a manner such 
that $g$ becomes the standard Euclidean metric plus terms that fall off sufficiently rapidly at infinity. 
More generally, a Riemannian $n$-manifold $(M,g)$ is said to be {\em asymptotically locally Euclidean} (or {\em ALE}\,) 
if the complement of a compact set $\mathbf{K}$ consists of finitely many components, each of which 
is diffeomorphic to  a quotient $(\RR^n-\mathbf{D}^n)/\Gamma_j$, where    $\Gamma_j\subset {\mathbf O}(n)$ is a finite subgroup 
which acts freely on the unit sphere,
in such a way that $g$ again becomes the Euclidean metric plus error terms that fall off sufficiently rapidly at infinity.
The components of $M-\mathbf{K}$ are called the {\em ends} of $M$;  their fundamental groups are 
the afore-mentioned groups $\Gamma_j$, which may in principle  be different
for different ends of the manifold. 

The {\em mass} of an ALE  Riemannian $n$-manifold is an invariant which assigns a real number
to each end. This concept originated in general relativity, where an asymptotically flat $3$-manifold
could be interpreted as 
representing a time-symmetric slice of some $4$-dimensional space-time, in which case 
this invariant becomes  the 
 so-called  ADM mass \cite{adm},  which reads off  the apparent mass of an isolated gravitational 
source  from the asymptotics of its gravitational field. Our conventions are chosen so that, 
at a given end,
the mass of an ALE manifold is given by 
$$
{\zap m}(M, g) := \lim_{\varrho\to \infty}  
 \frac{\mathbf{\eg} (\frac{n}{2})}{4(n-1)\pi^{n/2}} \int_{S_\varrho/\Gamma_j} \left[ g_{k\ell, k} -g_{kk,\ell}\right] \mathbf{n}^\ell d\mathfrak{a}_E
$$
where commas represent derivatives in the given asymptotic coordinates, 
summation over
repeated indices is implicit, 
$S_\varrho$ is the Euclidean coordinate sphere of radius $\varrho$,  
$d\mathfrak{a}_E$ is the $(n-1)$-dimensional volume form induced on this sphere by the Euclidean 
metric,  and $\vec{\mathbf{n}}$ is the outward-pointing Euclidean unit normal vector.
While our choice here of normalization factor is of course primarily a matter of convention, 
an explanation of this choice is provided in the Appendix. Perhaps the most controversial feature of
our definition is that we have
 specified  that the integral is to be taken over $S_\varrho/\Gamma_j$ rather than over $S_\varrho$, 
so that the mass, by our conventions, is $1/|\Gamma_j|$ times the value
one  might  otherwise  expect. 

Needless to say, this peculiar definition of the mass seems to  depend on the choice of asymptotic coordinates. Indeed, 
 without additional assumptions, the relevant limit might not even exist, or might be
coordinate dependent. However, 
 Bartnik \cite{bartnik} and Chru\'{s}ciel \cite{admchrus} independently discovered that the mass is finite and independent of the choice of
 asymptotic coordinates provided we impose weak fall-off conditions of the following type:
 \label{conditions}
 \begin{enumerate}[(i)]
 \item
 \label{eins}
  the scalar curvature $s$ of  the $C^2$ metric  $g$ belongs to $L^1$;  and
\item
\label{zwei}
 in  some asymptotic chart at each end of $M^n$, the components of the metric 
satisfy $g_{jk} - \delta_{jk} \in C^{1,\alpha}_{-\tau}$ for  some $\tau > (n-2)/2$  and   some $\alpha\in (0,1)$.
 \end{enumerate}
 Here  the weighted H\"older spaces  $C^{k,\alpha}_{-\tau}$  consist of $C^{k,\alpha}$ functions such 
 that 
 $$
 \left(\sum_{j = 0}^{k} |x|^{j} |\triangledown^j f(x)|\right) + |x|^{k+\alpha} [\triangledown^k f]_{C^{0,\alpha}(B_{|x|/10}(x))} = O(|x|^{-\tau}).
 $$
 This definition can naturally be extended to tensor fields, and the resulting $C^{k,\alpha}_{-\tau}$ spaces
then become Banach spaces when equipped with the obvious weighted analogs of the usual H\"older norms. 
 While Bartnik  actually does mention  these weighted H\"older spaces in passing \cite[Theorem 1.2 (v)]{bartnik},   the state of the literature at the time led
  him to  instead impose a  slightly stronger condition  in lieu of   (\ref{zwei}), by instead  requiring $g-\delta$ to 
 belong to the weighted Sobolev spaces $W^{2,q}_{-\tau}$ for some $q > n$ and some $\tau > (n-2)/2$. 
 Bartnik's condition implies (\ref{zwei}), and condition (\ref{zwei})
 in turn  implies 
that, for some $\varepsilon > 0$,  the metric $g$ satisfies the Chru\'{s}ciel-type fall-off condition 
$$
g_{jk} = \delta_{jk} + O (|x|^{1-\frac{n}{2}-\varepsilon}), \qquad g_{jk,\ell} = 
O (|x|^{-\frac{n}{2}-\varepsilon}) 
$$
in suitable coordinates;  and 
 this  Chru\'{s}ciel-type fall-off is actually all that is needed  for many of our key results. 
The central issue is really  the range of  fall-off rates $\tau$ that are to be allowed; as emphasized  by both Bartnik and Chru\'{s}ciel, 
allowing   slower rates of fall-off than indicated above would make the mass coordinate-dependent, and so essentially ill-defined. 
Our  definition of  an ALE manifold 
 will  therefore  {\em by default} include
conditions (\ref{eins}) and (\ref{zwei}), except where we clearly specify that a weaker assumption suffices for a given result. 
When $n=4$, some our proofs will also require analogous control of  an  extra derivative of the 
the metric, and so {\em by default} we will strengthen assumption (\ref{zwei}) in this special dimension to instead require 
that $g_{jk}-\delta_{jk} \in C^{2,\alpha}_{-\tau}$ for some $\tau > (n-2)/2=1$,
although we will  also  sometimes explicitly  weaken this assumption when it is not needed for a given result.

The coordinate-based definition of the mass makes it seem  both enigmatic and chimerical.  In this article, however, we will show  that the mass 
 has a completely transparent  meaning 
when the  ALE space in question is a {\em K\"ahler} manifold. 
Along the way, we will incidentally learn  that an ALE  K\"ahler 
 manifold   only has one end; thus, in the K\"ahler setting,  a  choice of end is not required  in order  to be able to discuss the mass
in the first place!

Rather than beginning with  general ALE K\"ahler manifolds, let us first highlight 
the setting that originally motivated our investigation:
the so-called scalar-flat case, 
where the scalar curvature is assumed to vanish identically. 
In this context, we will demonstrate  the following result:

\begin{main}\label{alpha}
The mass of an ALE scalar-flat K\"ahler manifold $(M,g,J)$ is a topological invariant,
determined entirely by the smooth manifold $M$, together with the first Chern class
$c_1= c_1(M,J)\in H^2(M)$ of the complex structure and the K\"ahler class 
$[\omega]\in H^2(M)$ of the metric.
\end{main}

In fact, our proof actually provides an  {\em explicit} formula for the mass in terms of these data.
Revisiting familiar  examples, this in particular gives a pure-thought explanation of 
 the second author's observation \cite{lpa} 
that there are ALE scalar-flat K\"ahler surfaces\footnote{Throughout the article, we use the 
term {\em complex surface} to indicate 
a complex 
manifold of {\em complex} dimension $2$, and thus of real dimension $4$.} of negative mass. Rather 
more interestingly, though, a quick glance at  other known examples
immediately now  gives a negative answer\footnote{We would like to thank Ioana \c{S}uvaina for  
pointing out to us
that this answer was already  implicit in   results of  Rollin-Singer  \cite[\S 6.7]{rosimass}
regarding the toric case.} 
to a  question posed by Arezzo \cite{arezzomass} that naturally arose in connection with  
gluing  constructions for  cscK metrics:

\begin{main}\label{beta}
There are infinitely many topological types of ALE scalar-flat K\"ahler surfaces 
that have zero mass, but are not Ricci-flat. 
\end{main}
\noindent
By contrast,   Corollary 
\ref{thespot} below,
which was pointed out to us by Cristiano Spotti,  gives a systematic explanation of why the mass actually  turns out to  be negative
for so many other concrete examples. 

We now come to the actual formula for the mass. Because $M$ is a 
smooth manifold, one can define the compactly supported  de Rham cohomology 
$H^k_c(M)$, as well as the usual de Rham cohomology. If $M$ is a complex manifold, 
it is in particular oriented, and Poincar\'e duality  therefore gives us  an isomorphism 
$H^2_c(M)\cong[H^{2m-2}(M)]^*$. On the other hand, there is a natural 
map $H^2_c(M)\to H^2(M)$ induced by the inclusion of compactly supported forms
into all differential forms, and in the ALE setting, this map is actually an {\em isomorphism}.
We may therefore 
define 
$$\clubsuit: H^2(M) \to H^2_c(M)$$
to be its inverse. Using this notation, we may now state our explicit formula for the mass:

\begin{main}\label{gamma}
 Any ALE K\"ahler manifold  $(M,g,J)$ of complex dimension $m$ has mass 
given by 
$${\zap m}(M,g) = - \frac{\langle \clubsuit (c_1) , [\omega ]^{m-1}\rangle}{(2m-1)\pi^{m-1}} +
\frac{(m-1)!}{4(2m-1)\pi^m} \int_M s_g d\mu_g$$
where $s_g$ and $d\mu_g$ are respectively  the scalar curvature and volume form of $g$, 
while  $c_1 =c_1(M,J)\in H^2 (M)$ is the first Chern class of the complex structure, 
$[\omega]\in H^2(M)$ is the
K\"ahler class of $g$, and $\langle~, ~\rangle$ is the duality pairing between $H^2_c(M)$ and 
$H^{2m-2}(M)$.
\end{main}
\noindent 
Here we remind the reader  that the Bartnik-Chru\'{s}ciel  fall-off condition (\ref{eins})  requires the scalar curvature $s$ to be integrable. 
It is remarkable that this feature\footnote{As explained to us by 
Gustav 
Holzegel,  the intimate relationship between mass and scalar curvature 
apparently first came to light 
in the work of
Brill \cite{brill} on  stationary axisymmmetric space-times.} plays a direct role in our setting, by ensuring that 
the second term on the right-hand side  is well-defined.  
It is worth noting that our discussion  will not simply  
rely on the Bartnik-Chru\'{s}ciel theorem on the coordinate-invariance of the mass, but rather 
will actually provide  an independent verification of it in the K\"ahler setting. 

The reader may find it illuminating to compare Theorem \ref{gamma} with the more familiar compact case.
If $(M^{2m},g,J)$ is a {\em compact} K\"ahler manifold of complex dimension $m$, then its total scalar curvature is well known to be   topologically
 determined \cite{bes,calabix}  by the first Chern class of the complex structure   and the K\"ahler class of the metric via the Gauss-Bonnet-type formula 
$$\int_M s~d\mu= \frac{4\pi}{(m-1)!} \langle c_1, [\omega ]^{m-1}\rangle .$$ 
The gist of Theorem \ref{gamma} is that the mass  measures the degree to which this formula
fails  in the ALE case:
$$\frac{4\pi^m(2m-1)}{(m-1)!}{\zap m}(M,g) =   \int_M s~d\mu- \frac{4\pi}{(m-1)!} \langle \clubsuit c_1, [\omega ]^{m-1}\rangle.$$
In other words, the mass may be understood as  an {\em anomaly} in the formula for the total scalar curvature, encapsulating an 
essential difference between the ALE and compact cases. 

Of course, the formula in Theorem \ref{gamma} simplifies when 
 $g$ is scalar-flat; the integral on the right  drops out, and the mass is then 
expressed  purely in terms of topological data.  Theorem \ref{alpha} is thus an immediate corollary. 
Theorem \ref{beta}  is then proved  by applying this  formula to some 
ALE scalar-flat K\"ahler surfaces  constructed  by the second author in \cite{mcp2}. 

As we've already  noted, there are ALE manifolds of non-negative scalar curvature
which nonetheless have negative mass. However, one expects  this to  never happen for
AE (asymptotically Euclidean) manifolds. Indeed, this is  actually  a theorem \cite{lp,syaction,witmass}
if one is willing to further assume that the manifold is either low-dimensional or spin. Here we can 
add something new to the discussion, by demonstrating  that the conjecture
also  always holds in the K\"ahler case, even if the manifold is  high-dimensional and
non-spin:

\begin{main}[Positive Mass Theorem for K\"ahler Manifolds]
\label{delta}
Any asymptotically Euclidean (AE) K\"ahler manifold with non-negative scalar curvature
has non-negative mass:
$$s\geq 0 \quad \Longrightarrow \quad {\zap m}(M,g) \geq 0.$$
Moreover, ${\zap m}(M,g) = 0$ in this context  iff $(M,g)$ is Euclidean space.
\end{main}

Our proof of this version of the positive mass theorem uses nothing but our mass formula and
some complex manifold theory. Indeed, the argument actually tells us a great deal more; it
in fact shows that the mass can be bounded from below by the $(2m-2)$-volume
of a subvariety. This is reminiscent of  the Penrose inequality \cite{braymass,huilmpen,penineq}, 
which gives a sharp lower bound for  the mass of an AE $3$-manifold in terms
of the area of a minimal surface. Our K\"ahler analog goes as  follows:

\begin{main}[Penrose Inequality for K\"ahler Manifolds]
\label{epsilon}
Let $(M^{2m},g,J)$ be an AE K\"ahler manifold with   scalar curvature $s \geq 0$.
Then $(M,J)$ carries a  canonical divisor $D$ that is expressed as a sum $\sum n_jD_j$
of compact complex hypersurfaces with  positive integer coefficients, with the property that
 $\bigcup_j D_j \neq \varnothing$  
whenever  $(M,J)\neq \CC^m$. 
In  terms of this divisor, we then have  
$${\zap m}(M,g) \geq  \frac{(m-1)!}{(2m-1)\pi^{m-1}} \sum_j  n_j \mbox{Vol}\, (D_j) 
$$
and equality holds if and only if $(M,g,J)$  is scalar-flat K\"ahler. 
\end{main}

Much of  the intrinsic interest of our subject arises from the case of  real dimension $4$,
where a plethora of known examples leads to  a wealth of 
  applications, including Theorem \ref{beta}. However, 
the complex-surface case 
entails   technical subtleties 
that simply disappear  in higher dimensions.
Our presentation therefore begins  
with proofs of Theorems \ref{alpha} and \ref{gamma} in complex dimension $m \geq 3$. 
Using this high-dimensional case as our guide, but now emphasizing  the coordinate-invariant nature  of the mass, 
we then develop  a second, more robust proof of the  asymptotic form of our  mass formula, 
in a manner that also shows that this formula remains valid 
in complex dimension $2$. We then prove some global results regarding ALE K\"ahler surfaces, 
culminating in a proof of the $m=2$ case of Theorem \ref{gamma}, 
along with various applications, including Theorem \ref{beta}. 
We then conclude by showing that  Theorems \ref{delta} and \ref{epsilon}
are straightforward corollaries of  our other results.

\section{The High-Dimensional Case} \label{hiyo}

We begin by proving Theorems \ref{alpha} and \ref{gamma}  when the complex dimension is $m\geq 3$. 
Our high-dimensional proofs will  prefigure 
many  of the  ideas needed for the complex-surface ($m=2$) case, but  manage to avoid a number of difficult  technical complications. 
 Our journey begins with the following step:

\begin{lem}\label{ndn}
Let ${M}_\infty$  be an end of an ALE K\"ahler manifold $(M^{2m},g,J)$, $m\geq3$, 
and let $(x^1, \ldots , x^{2m})$ be a real asymptotic coordinate system on the universal cover 
 $\widetilde{M}_\infty$ of $M_\infty$ in  which $g$ satisfies the weak fall-off hypothesis 
 $$g_{jk}= \delta_{jk} + O(|x|^{1-m-\varepsilon}), \qquad g_{jk,\ell}= O(|x|^{-m-\varepsilon})$$
 for some $\varepsilon > 0$. 
 Then there is a (non-compact) complex $m$-manifold $\mathscr{X}$
 containing an embedded complex hypersurface $\Sigma\cong \CP_{m-1}$ with  normal bundle of degree $+1$, such that 
 $\widetilde{M}_\infty$  is biholomorphic to $\mathscr{X}-\Sigma$. 
\end{lem}
\begin{proof}
We  first identify the range  $\RR^{2m} - \mathbf{D}^{2m}$ of our asymptotic coordinate system with $\CC^m - \mathbf{D}^{2m}$
 in a reasonably intelligent manner, 
 by choosing a constant-coefficient almost-complex structure $J_0$ on $\RR^{2m}$ such that $J\to J_0$ at infinity. 
 We can do this by identifying all the tangent spaces of $\RR^{2m}$ in the usual way, using the flat   Euclidean 
 connection $\triangledown$. 
  Since $\nabla J=0$, where $\nabla$ is the Levi-Civita connection of $g$, 
 and since $\nabla = \triangledown + \boldsymbol{\Gamma}$, where $\triangledown$ is the coordinate Euclidean connection and
  $\boldsymbol{\Gamma}= O(\varrho^{-m-\varepsilon})$, 
 the value of $J$ will approach a well-defined limit $J_0$ along some chosen radial ray,
 and we then extend this limit as a constant-coefficient tensor field on our asymptotic coordinate
 domain. Along the chosen ray, we then have $J-J_0 = O(\varrho^{1-m-\varepsilon})$, and integrating along
 great circles in spheres of constant radius  then shows that $J-J_0$ has  $O(\varrho^{1-m-\varepsilon})$
 fall-off everywhere. The same argument  similarly shows that the derivative of
 $J$ falls off at  the same rate  as the derivative of the metric $g$. 

Now think of $(\CC^m, J_0)$ as  an affine chart on $\CP_m$, whose complex structure we will
 also denote by $J_0$. Let $\Sigma \subset \CP_m$ be the hyperplane at infinity, and notice that 
  our asymptotic coordinates give us a diffeomorphism between $\widetilde{M}_\infty$ and $\mathscr{X}-\Sigma$,  
  where $\mathscr{X}\subset \CP_m$ is   some neighborhood 
 of this hyperplane. We may then define a ``rough'' almost complex structure
  $J$ on $\mathscr{X}$ by taking it to be the given $J$ on $\mathscr{X}-\Sigma$, and $J_0$ along $\Sigma$. This 
  $J$ is  then at least $C^1$  on $\mathscr{X}$. Indeed, if $(z^1, z^2,  \ldots , z^m)$ are the standard affine coordinates 
  on $\CC^m$, then, in the cone $|z^1| \geq \max \{   |z^j|~|~j> 1\}$,  we may inspect the behavior of $J$ near infinity
  by observing that there is a unique $(m,0)$-form with respect to $J$ given by  
  $$\varphi = (dz^1+ \varphi_1^{\bar{j}} d\bar{z}^j ) \wedge (dz^2+  \varphi_2^{\bar{j}} d\bar{z}^j) \wedge \cdots \wedge (dz^m+  \varphi_m^{\bar{j}} d\bar{z}^j),$$
  and that the functions $\varphi^{\bar{j}}_k$ then have the  same fall-off as $J$.   Setting 
  $$(w_1,w_2, \ldots, w_m ) = \left(\frac{1}{z^1}, \frac{z^2}{z^1}, \ldots , \frac{z^m}{z^1}\right),$$
   one  can then analogously determine the 
  components of $J$ from those of 
  $$\psi := -w_1^{m+1} \varphi  = (dw_1+ \psi_1^{\bar{j}} d\bar{w}_j ) \wedge (dw_2+  \psi_2^{\bar{j}} d\bar{w}_j) \wedge 
  \cdots \wedge (dw_m+  \psi_m^{\bar{j}} d\bar{w}_j).$$
  Reading off the coefficients $\varphi^{\bar{j}}_k$ and $\psi^{\bar{j}}_k$ 
  by inspecting the  type $(m-1,1)$ parts of   $\varphi$ and $\psi$ with respect to the background complex structure $J_0$, 
  one then sees that  the coefficients $\{ \psi^{\bar{j}}_k\}$  behave like the $\{ \varphi^{\bar{j}}_k\}$ times, at worst, 
  $O(|w_1|^{-1})$, while their  first derivatives   behave like those of the $\{ \varphi^{\bar{j}}_k\}$ times, 
  at worst, $O(|w_1|^{-3})$. Since $\varrho^{-1}= O(|w_1|)$ in the region in question, our fall-off conditions therefore guarantee
   that the almost-complex structure is at least 
  $C^1$. In particular, the Nijenhuis tensor of $J$ is continuous, and since it vanishes
 on the dense set $\mathscr{X}-\Sigma$, it vanishes identically. The 
 Hill-Taylor version \cite{hiltay} of Newlander-Nirenberg therefore guarantees
 the existence of  complex coordinates on $(\mathscr{X}, J)$. These will  at least have H\"older regularity $C^{1,\alpha}$
with respect to the original atlas, for any $\alpha\in (0,1)$. 
\end{proof}

In fact,  our fall-off conditions are noticeably   stronger than what is actually 
needed for  the proof of
this lemma. In any case, whenever we can add such 
a hypersurface at infinity, then, provided the 
 complex dimension is $m\geq 3$, the following result will  force the 
complex structure $J$  to  become  completely {\em standard} at infinity:

\begin{lem}\label{nostalgia}
Let $({\mathscr{X}},J)$ be a (possibly non-compact) 
complex $m$-manifold, $m\geq 3$, that contains an embedded hypersurface
$\Sigma\subset \mathscr{X}$ which  is biholomorphic to $\CP_{m-1}$ and 
has  normal bundle of degree $+1$. 
Then $\Sigma \subset \mathscr{X}$ has an open neighborhood $\mathscr{U}$   which is biholomorphic to
an open neighborhood of a hyperplane $\CP_{m-1}\subset \CP_m$. 
\end{lem}
\begin{proof} Since the normal bundle of $\Sigma\cong \CP_{m-1}$ is isomorphic to $\mathcal{O}(1)$, and 
since $H^1(\CP_{m-1}, \mathcal{O}(1))=0$, a theorem of Kodaira \cite{kodsub} 
implies that there is a complete analytic family of compact complex submanifolds of dimension
$h^0(\CP_{m-1}, \mathcal{O}(1))=m$ which represents all small deformations of $\Sigma
\subset \mathscr{X}$ 
through compact complex submanifolds. Since $\CP_{m-1}$ is rigid, and since $h^{0,1}(\CP_{m-1})=0$,
we may assume, by shrinking the size of the family if necessary, that every submanifold in the family 
is biholomorphic to $\CP_{m-1}$, and has normal bundle $\mathcal{O}(1)$. Let us use $\mathscr{Y}$ to denote the
complex $m$-manifold which parameterizes these hypersurfaces; and for any 
$y\in \mathscr{Y}$, let  $\Sigma_y\subset \mathscr{X}$ be the corresponding complex hypersurface. Note that, 
by construction, $\Sigma = \Sigma_o$ for some base-point $o\in \mathscr{Y}$. 

Now 
Kodaira's theorem also gives us a natural identification of the tangent space $T^{1,0}_y\mathscr{Y}$ 
with the  holomorphic sections of the normal bundle of $\Sigma_y\subset \mathscr{X}$. Since 
$H^0(\CP_{m-1}, \mathcal{O}(1))$ consists of linear functions on $\CC^m$, the space of
complex directions $\PP (T^{1,0}_y\mathscr{Y})$ can thus  be naturally  identified with the space
of hyperplanes $\CP_{m-2}\subset \Sigma_y\cong \CP_{m-1}$; in other words, each 
 $\Sigma_y$ is exactly the dual projective space $\PP (\Lambda^{1,0}_y\mathscr{Y})$ of the projectivized tangent space $\PP (T^{1,0}_y\mathscr{Y})$. 

This leads us to consider the space $\mathscr{Z}$ of those  embedded $\CP_{m-2}$'s in $\mathscr{X}$
that arise as hyperplanes in the various  $\Sigma_y$.
Thus, by definition,  each $z\in \mathscr{Z}$
 corresponds to  a submanifold $\Pi_z\cong \CP_{m-2}$ of 
$\mathscr{X}$. But since $\Sigma_y=\PP(\Lambda^{1,0}_y\mathscr{Y})$, any point
of the projectivized tangent bundle $${\zap p}: \PP(T^{1,0}\mathscr{Y} )\to \mathscr{Y}$$ also gives
rise to some such submanifold $\Pi_z$. 
Since any $\Pi_z\cong \CP_{m-2}$ has normal bundle
$\mathcal{O}(1) \oplus \mathcal{O}(1)$, 
 the family
$\Pi_z$, $z\in \mathscr{Z}$, is therefore complete in the sense of Kodaira, because every section
of the normal bundle $\mathcal O(1)\oplus \mathcal{O}(1)$ can be realized by some variation in 
$\PP (T^{1,0}\mathscr{Y})$. In fact, this observation actually tells us a great deal more; not only is $\mathscr{Z}$ a complex manifold of complex
dimension $2m-2$, but there is a natural surjective holomorphic submersion 
${\zap q}: \PP (T^{1,0}\mathscr{Y})\to {\mathscr Z}$. We thus obtain a   double fibration 
\setlength{\unitlength}{1ex}
\begin{center}\begin{picture}(20,17)(0,3)
\put(10,17){\makebox(0,0){$\PP(T^{1,0}\mathscr{Y})$}}
\put(2,5){\makebox(0,0){$\mathscr{Z}$}}
\put(18,5){\makebox(0,0){$\mathscr{Y}$}}
\put(15,12){\makebox(0,0){${\zap p}$}}
\put(5,12){\makebox(0,0){${\zap q}$}}
\put(11,15.5){\vector(2,-3){6}}
\put(9,15.5){\vector(-2,-3){6}}
\end{picture}\end{center}
which embeds $\PP (T^{1,0}\mathscr{Y})$ into the product ${\mathscr Z}\times {\mathscr Y}$,
and thereby realizes it as
$$\PP (T^{1,0}\mathscr{Y})= \{ (z,y) \in {\mathscr Z}\times {\mathscr Y}~|~\Pi_z \subset \Sigma_y\}.$$
In particular, for any $z\in {\mathscr Z}$, the curve $\gamma_z\subset
\mathscr{Y}$ given by ${\zap p}[{\zap q}^{-1}(z)]$ exactly consists of those $y\in {\mathscr Y}$ for
which $\Sigma_y\supset \Pi_z$. But this also shows  that $\gamma_z$ is an immersed complex 
curve, with tangent space at $y$ exactly consisting of sections of the normal bundle $\mathcal{O}(1)$
of $\Sigma_y\cong \CP_{m-1}$ which 
vanish at the hyperplane $\Pi_z\cong \CP_{m-2}$. Hence 
 the lift $\tilde{\gamma}_z\to \PP (T^{1,0}\mathscr{Y})$
of  $\gamma_z$ defined by $\tilde{\gamma}_z:=T^{1,0}\gamma_z$ coincides with ${\zap q}^{-1}(z)$. 
In particular, the holomorphic system of complex curves $\gamma_z$, $z\in {\mathscr Z}$, has the property
that there is exactly one such curve tangent to each direction in ${\mathscr Y}$. 
 By \cite[Proposition 1.2.I]{lebthes},  the curves $\gamma_z$ are therefore the 
unparameterized geodesics of a unique holomorphic projective connection on $\mathscr{Y}$; moreover,
by replacing  $\mathscr{Y}$ with a smaller neighborhood of $o$  if necessary, we may arrange that this projective connection 
is globally represented by some torsion-free holomorphic affine connection 
$\nabla$, with respect to which $\mathscr{Y}$  is geodesically convex. 
This allows us to identify $\mathscr{Z}$ with the space of unparameterized complex geodesics of $\nabla$.

Let $\mathscr{U}\subset \mathscr{X}$ be the open set defined by 
$$\mathscr{U} = \bigcup_{y\in \mathscr{Y}}\Sigma_y.$$
The fact that this is open follows from the fact that  the normal bundle $\mathcal{O}(1)$ of every $\Sigma_y$
is everywhere generated by its global sections.
But now, by construction,   every $x\in \mathscr{U}$ belongs to 
 $\Sigma_y$ for some $y\in \mathscr{Y}$. For each $x\in \mathscr{U}$, we can therefore  
define  a non-empty hypersurface 
$\mathscr{S}_x\subset \mathscr{Y}$  by 
$$
\mathscr{S}_x:=\{ y\in \mathscr{Y}~|~ x\in \Sigma_y\}.
$$
This is a non-singular hypersurface,  because the normal bundle of each $y\in \mathscr{S}_x$
has a global holomorphic section which is non-zero at $x\in \Sigma_y$;  the set of 
normal sections vanishing at $x$ thus has  complex codimension $1$, and exactly 
corresponds to  $T^{1,0}_y\mathscr{S}_x\subset T^{1,0}_y\mathscr{Y}$. Moreover, 
since  $\Sigma_y= \PP(\Lambda^{1,0}_y\mathscr{Y})$, 
the tangent space $T^{1,0}_y\mathscr{S}_x$, for any 
 $x\in \Sigma_y$,  is exactly the 
hyperplane in $T^{1,0}_y\mathscr{Y}$ annihilated by the $1$-dimensional subspace 
$x\subset \Lambda^{1,0}_y\mathscr{Y}$. It follows that  there is a hypersurface
 $\mathscr{S}_x$ tangent to any given hyperplane in $T^{1,0}\mathscr{Y}$.

 However,  with respect to $\nabla$, the 
 hypersurfaces  $\mathscr{S}_x$ are all {\em totally geodesic!}
 Indeed, if $y\in \mathscr{S}_x$ and $\xi\in T^{1,0}_y\mathscr{S}_x-0$, the section of the normal 
 bundle of $\Sigma_y$ which represents $\xi$ must vanish at $x$, and must do so at some 
 $\Pi_z\cong \CP_{m-2}$ containing $x$. The geodesic $\gamma_z$ through $y$ in the direction $\xi$
 therefore precisely consists of those $y^\prime \in \gamma_z\subset \mathscr{Y}$ for which
   $\Pi_z\subset \Sigma_{y^\prime}$. But since $x\in \Pi_z$, we therefore have
   $x\in \Sigma_{y^\prime}$ for every $y^\prime \in \gamma_z$, and it therefore follows that 
   $\gamma_z\subset \mathscr{S}_x$. This shows that $\mathscr{S}_x$  is totally geodesic,
   as claimed.

However, a classical theorem of Schouten and Struik \cite[p. 182]{schoustr} asserts that a projective
connection  in dimension $m\geq 3$ is projectively flat iff every hyperplane element is tangent to 
a totally geodesic hypersurface; cf. \cite[p. 290]{schouten}. Thus $\nabla$ is projectively flat, 
and $o\in {\mathscr Y}$ therefore has a  neighborhood  which 
can be identified with a ball in $\CC^m$, in such a manner that the unparameterized 
geodesics of $\nabla$ are just the intersections of  complex lines in $\CC^m$
with the ball. Let us again shrink $\mathscr{Y}$ by replacing it with this ball about $o$. 
The hypersurfaces $\mathscr{S}_x$ are now just the intersections of hyperplanes in $\CC^m$
with the ball $\mathscr{Y}$; in other words, thinking of $\CC^m$ as an affine chart on $\CP_m$,
they are just the intersections of projective hyperplanes with a fixed ball about $o$. For
the smaller ${\mathscr U}$ that corresponds to this smaller $\mathscr{Y}$, we therefore 
get an injective  holomorphic map to the dual projective space $\CP_m^*$ by sending $x\in {\mathscr U}$ to 
the hyperplane which intersects ${\mathscr Y}$  in ${\mathscr S}_x$. This provides
 the promised biholomorphism
between ${\mathscr U}\supset \Sigma_o=\Sigma$ and a neighborhood of a hyperplane in 
projective $m$-space. 
\end{proof}

\begin{rmk}
The above-cited result of Schouten and Struik  is proved by  showing that the Weyl projective
curvature of the projective connection vanishes, and then using the  fact, due to Weyl 
\cite[p. 105]{weylproj},
that, when $m\geq 3$, 
 this curvature condition is equivalent to the projective connection being projectively flat. 
 The fact that this fails when $m=2$ gives the complex surface case an entirely different flavor, as we
 will see in Lemma \ref{aria} below. 
 
There  are certainly  many other ways of proving the above result.
One alternative  strategy would proceed by first using  \cite{thick} to  show
 the infinitesimal neighborhoods of $\Sigma \subset {\mathscr X}$ are all 
standard, and then invoking  \cite{prinzip} or \cite{formal} to conclude that a neighborhood
of $\Sigma\subset \mathscr{X}$ is therefore biholomorphic to a neighborhood of $\CP_{m-1}\subset \CP_m$. 
\end{rmk}

Perhaps the single  most important consequence of  Lemma \ref{nostalgia} is that $J$ must always be standard at infinity  when $m\geq 3$.
For us, it is vital that the asymptotic coordinates which put $J$  in   standard form can moreover 
be chosen to be consistent with the hypothesized fall-off of the metric: 

\begin{lem} \label{standard} 
Let $(M^{2m},g,J)$ be an ALE K\"ahler manifold of complex dimension $m\geq 3$ which, in some real coordinate system on each end,
 merely satisfies condition {\rm (\ref{zwei})}, as set out on page \pageref{conditions} above. 
Then there are asymptotic  complex coordinates $(z^1, \ldots , z^m)$  on the universal  cover of $\widetilde{M}_\infty$ of any  end, 
 in which the complex structure $J$ becomes the standard one on $\CC^m$,
and in which the metric has fall-off 
$$g_{jk} = \delta_{jk} + O(|z|^{1-m-\varepsilon}), \quad  g_{jk,\ell} = O(|z|^{-m-\varepsilon})$$
for some $\varepsilon > 0$. 
\end{lem}
\begin{proof} Let $(\tilde{x}^1, \ldots , \tilde{x}^{2m})$ be some given  asymptotic  coordinate system in which $g_{jk}-\delta_{jk} \in C^{1,\alpha}_{-\tau}$
for some $\tau > m-1$ and some $\alpha\in (0,1)$, and let us once again set $\varepsilon = \min (\tau - (m-1), \alpha )$. 
We now think of   $\RR^{2m}$, with real coordinates $(\tilde{x}^1, \ldots , \tilde{x}^{2m})$ and the constant-coefficient almost-complex structure $J_0$
of the proof of Lemma \ref{ndn}, as an affine
 chart on  $\CP_m$. 
Lemma \ref{nostalgia}, in conjunction with the proof of Lemma \ref{ndn},  
then gives us  a $C^{1}$ diffeomorphism $\Psi$ between neighborhoods of $\CP_{m-1}\subset \CP_m$ that,  by \cite{nijenwoo,malgrange},  restricts as  a $C^{2,\varepsilon}$ diffeomorphism 
between the complement of a compact set in $\RR^{2m}$, with coordinates $\tilde{x}$,  and   the complement of a compact set in $\CC^m$, equipped with 
standard  complex coordinates  $(z^1, \ldots , z^m)$; and let $(x^1, \ldots , x^{2m})$ be the real 
and imaginary parts of $(z^1, \ldots , z^m)$, so that $z^j= x^{2j-1}+ i x^{2j}$. 
Now  notice that, for some large constant $C$,  we in particular have
$$C^{-1} |x| < |\tilde{x}|< C|x|$$
outside a large ball, simply because $\Psi$ is  uniformly Lipschitz near $\CP_{m-1}\subset \CP_m$. 
Because $\Psi$ is by construction holomorphic with respect to 
the complex structure $J$, the functions $z^\mu:=\Psi^*z^\mu$ are holomorphic with respect to the complex structure  associated with our K\"ahler metric,
so their real and imaginary parts $x^j:=\Psi^*x^j$ are harmonic functions with respect to the K\"ahler metric $g$. 
We now use a  partition of unity to construct a $C^{1,\varepsilon}$ Riemannian metric $\bar{g}$  on 
$\RR^{2m}$ which coincides
with $g$ outside some large ball, and use a smooth cut-off function 
to construct $C^{2,\varepsilon}$ functions $f^j$ on $\RR^{2m}$ which coincide with the $x^j:= \Psi^*x^j$ outside this same ball.
The Laplacians $\Delta_{\bar{g}}f^j$  of these functions are then compactly supported $C^{0,\varepsilon}$ functions
on $\RR^{2m}$. 

The fall-off of the first derivative of $g$ in $\tilde{x}$-coordinates implies  that
$$\Delta_{\bar{g}} \tilde{x}^j = g^{k\ell}{\mathbf \Gamma}_{k\ell}^j\in  C^{0,\varepsilon}_{-m-\varepsilon}(\RR^{2m}).$$
On the other hand, since $-m - \varepsilon \in (-2m, -2)$, the Laplacian $\Delta_{\bar{g}}$ induces an isomorphism \cite[Theorem 8.3.6]{joycebook}
between $C^{2,\varepsilon}_{2-m -\varepsilon}(\RR^{2m})$
and $C^{0,\varepsilon}_{-m-\varepsilon}(\RR^{2m})$. Thus, for each $j$, there is 
 a unique  $u^j \in C^{2,\varepsilon}_{2-m -\varepsilon}(\RR^{2m})$
with $\Delta_{\bar{g}} u^j = \Delta_{\bar{g}}\tilde{x}^j$.
The functions  $\tilde{y}^j:=\tilde{x}^j- u^j$ are then $\bar{g}$-harmonic functions 
 on $\RR^{2m}$, and  provide  coordinates at infinity   that 
are asymptotic to the $\tilde{x}^j$. But,
 since $\Delta_{\bar{g}}f^j\in C^{0,\varepsilon}_\beta$ for $\beta< -2m$, \cite[Theorem 8.3.6]{joycebook} also asserts that there is, for each $j$, a unique 
function $v^j\in C^{2,\varepsilon}_{2-2m}$  with $\Delta_{\bar{g}}v^j= \Delta_{\bar{g}}f^j$. The functions $y^j=f^j-v^j$ are then yet another
set of $\bar{g}$-harmonic functions which give us coordinates at infinity,  this time  instead asymptotic to the  $x^j$. 

Now choose some $\eta\in (1,2)$ and some   $q > 2m$.
Since   the $\tilde{y}^j$ and the $y^j$ are  $O(|x|)=O(|\tilde{x}|)$ at infinity, they  therefore   belong to the weighted space 
$L^q_\eta$ used by Bartnik \cite{bartnik}.  On the other hand,  our $C^{1,\alpha}_{-\tau}$  fall-off condition on the metric guarantees  that 
$\bar{g}_{jk}-\delta_{jk} \in W^{1,q}_{1-m-\varepsilon/2}$ in $\tilde{x}$ coordinates,  so one of Bartnik's key results \cite[Theorem 3.1]{bartnik} now asserts that 
$$\mathcal{H}_{q,\eta} := \{ f\in L^q_\eta~|~ \Delta_{\bar{g}}f=0\}$$
has dimension $n+1=2m+1$, and hence   that 
$$\mbox{span}\, \{ 1,  \tilde{y}^1, \ldots  \tilde{y}^{2m}\} = \mathcal{H}_{q,\eta} = \mbox{span}\, \{ 1, y^1, \ldots , y^{2m} \}. $$
It follows that  the  $y^j$ are  affine-linear 
combinations of the $\tilde{y}^k$; in other words, 
$$y^j = a^j + A^j_k \tilde{y}^k$$
for an appropriate translation $\vec{a}\in \RR^{2m}$ and an appropriate invertible linear transformation $A\in \mathbf{GL}(2m, \RR)$. 
Consequently, 
$$x^j= a^j+ A^j_k\tilde{x}^k + w^j$$
outside a large ball, where 
$$w^j= v^j-A^j_ku^k\in C^{2,\varepsilon}_{2-m-\varepsilon}.$$
In particular, 
$\frac{\partial x^j}{\partial \tilde{x}^k}- A^j_k \in C^{1,\varepsilon}_{1-m-\varepsilon}$, and  inverting the Jacobian matrix then tells us that, as  functions of $\tilde{x}$, 
 $\frac{\partial \tilde{x}^k}{\partial {x}^j}- (A^{-1})^k_j\in C^{1,\varepsilon}_{1-m-\varepsilon}$.
We thus have 
$$\frac{\partial}{\partial x^j} = \left( A^k_j+ U^k_j\right) \frac{\partial}{\partial \tilde{x}^k}$$
where 
$U^k_j = \frac{\partial w^k}{\partial \tilde{x}^\ell} \frac{\partial \tilde{x}^\ell}{\partial x^j} \in C^{1,\varepsilon}_{1-m-\varepsilon}$.
In $\tilde{x}$ coordinates,  we therefore see that 
$$g(\frac{\partial}{\partial x^j}, \frac{\partial}{\partial x^k}) -  A^\ell_jA^\ell_k \in  C^{1,\varepsilon}_{1-m-\varepsilon}$$
and that 
$$\frac{\partial}{\partial x^\ell}  \left[ g(\frac{\partial}{\partial x^j}, \frac{\partial}{\partial x^k}) \right] \in  C^{0,\varepsilon}_{-m-\varepsilon}.$$
Since $C^{-1} |x| < |\tilde{x}| < C |x|$, this now immediately implies that  
$$
g_{jk} = (A^tA)_{jk} + O( |x|^{1-m-\varepsilon}), \quad g_{jk,\ell} =  O( |x|^{-m-\varepsilon})
$$
in $x$ coordinates. However, since the K\"ahler metric $g$ is Hermitian in  the complex coordinate system defined by the $z^j=x^{2j-1}+ ix^{2j}$, the matrix  $A^tA$
must represent a Hermitian inner product on $\CC^m$, and so can be written as $B^*B$ for some $B\in \mathbf{GL}(m, \CC)$. 
Thus, after a complex-linear change of coordinates $\vec{z}\mapsto B^{-1}\vec{z}$, 
we will then have 
$$
g_{jk} = \delta_{jk} + O( |z|^{1-m-\varepsilon}), \quad g_{jk,\ell} =  O( |z|^{-m-\varepsilon}),
$$
as desired. 
\end{proof}

\begin{rmk} The  above proof dovetails with Bartnik's  weighted-Sobolev results in a way that lets us avoid having to reinvent the wheel. However, 
we certainly could have  avoided passing to a complete manifold or citing Bartnik's count of harmonic functions of sub-quadratic growth. Indeed, 
the results in \cite[Chapter 6]{marshthes} allow one to argue directly that any  harmonic function on $(\widetilde{M}_\infty ,g)$  of polynomial growth is  asymptotic to 
a  harmonic polynomial on Euclidean space $(\RR^{2m}, \delta )$.
\end{rmk}

While Lemma \ref{standard} is still phrased in terms of {\em any} end, we will soon see  that there can actually  only  be  {\em one} end. 
Indeed, 
Lemma \ref{nostalgia}  opens up a thoroughfare   to  this and  other   global results,  
via  the following remarkable consequence:

\begin{lem} \label{compactify} 
Let $(M^{2m},g,J)$ be an ALE K\"ahler manifold of complex dimension $m\geq 3$. Then we may compactify
$(M,J)$ as a complex orbifold $(X,J_X)$ by adding a copy of $\CP_{m-1}/\Gamma_j$ to each end. Moreover, the resulting
complex orbifold {\bf admits  K\"ahler metrics}. 
\end{lem}
\begin{proof}
Lemma \ref{ndn} already told us that we could smoothly cap off the universal cover of each end by adding a $\CP_{m-1}$, 
and Lemma \ref{nostalgia} then showed that each such capped-off space is biholomorphic to a neighborhood $\mathscr{U}$ of $\CP_{m-1}\subset \CP_m$. Since 
the action of each $\Gamma_j$  extends continuously to $\mathscr{U}$, and since it is 
holomorphic on the complement of $\CP_{m-1}$, the induced action is actually holomorphic; and since 
Hartogs' theorem also tells us that this action extends holomorphically to all of $\CP_m$, 
$\Gamma_j$  therefore acts on $\mathscr{U}$ by projective linear transformations. This allows us to compactify
$(M,J)$ as a complex orbifold $(X,J_X)$  by adding a copy  of the appropriate $\CP_{m-1}/\Gamma_j$ to each end.
Here, of course,  $\Gamma_j$ is identified with a finite subgroup of $\mathbf{U}(m) \subset \mathbf{SU}(m+1)$, and so acts  
on a neighborhood of $\CP_{m-1}\subset \CP_m$ in a manner that preserves not only the complex structure, but also the standard  Fubini-Study metric. 

Using this  last observation, we will now 
construct a K\"ahler metric $\hat{g}$ on $(X,J_X)$. To do this, we first 
use  our asymptotic coordinates on the complement of a suitable  $\mathbf{K} \Subset M$ to identify the 
 universal cover of each  end  with the complement of a large
closed ball $\mathbf{D}^{2m}\subset \CC^m$ of radius $\varrho_0$ in a $\Gamma_j$-invariant manner. 
Since $\CC^m - \mathbf{D}^{2m}$  is $2$-connected, we can then write the K\"ahler form $\omega$ of our
given ALE K\"ahler metric $g$ as  
$$\omega = d (\beta+  \bar{\beta})  =  \partial  \beta    +   \bar{\partial} \bar{\beta}$$
for some $\bar{\partial}$-closed $(0,1)$-form $\beta$ on $\CC^m - \mathbf{D}^{2m}$. 
However, since $m\geq 3$, a result of Andreotti-Grauert \cite[p. 225]{andgrau} tells us that $H^{0,1}_{\bar{\partial}} ( \CC^m -\mathbf{D}^{2m})=0$.
 Thus  $\beta = \bar{\partial} h$ for some  function $h$, and we therefore have 
  $$\omega =  
  i \partial \bar{\partial} f$$
where $f= 2\, \Im m \, h$. By averaging over the action of $\Gamma_j$, we  then improve our choice  of  the potential   
  $f$ so as to make it $\Gamma_j$-invariant on each end. 
    
  We now introduce the function $u = \varrho^2 = \sum |z^j|^2$ on each end. If $F(u)$ is 
  any smooth function, then  along the $z^1$-axis we have 
  $$\partial \bar{\partial}F(u) = (u F^\prime )^\prime (u) dz^1 \wedge d\bar{z}^1 + F^\prime (u)\sum_{j= 2}^m dz^j \wedge d\bar{z}^j~,$$
  so that  ${\mathbf U}(m)$-invariance  implies that $i\partial \bar{\partial} F$ is positive semi-definite
  iff $uF^\prime (u)$ is a  non-negative, non-decreasing  function. Now  choose some radius $\varrho_1 > \varrho_0$, and 
  let $\psi (u)$ be a non-decreasing cut-off  function which is  $\equiv 0$ near $u = \varrho_0^2$ and  $\equiv 1$ for $u \geq \varrho_1^2$. 
 Let $F: [0,\infty)  \to [0,\infty)$ be the smooth function defined by 
 \begin{equation}
 \label{gadget}
F (u) = \int_0^u\frac{\psi (t)\, dt}{1+t} ,
\end{equation}
  so  that $uF^\prime (u)= \psi (u)\, [1-(1+u)^{-1}]$ is non-negative and non-decreasing. Since this ensures that  $i\partial \bar{\partial} F$ is positive semi-definite, 
  it follows that, for any constant $\mathbf{N}> 0$, $M$ admits  a K\"ahler metric $g_{\mathbf{N}}$ that equals $g/{\mathbf{N}}$ on the compact set
  $\mathbf{K}\subset M$, and  which has K\"ahler form given by 
  $$\omega_{\mathbf{N}}= \frac{\omega}{\mathbf{N}} + i \partial \bar{\partial}F = i \partial \bar{\partial}\left(F+\frac{f}{\mathbf{N}}\right)$$
 on the ends. In particular, since $i \partial \bar{\partial}F$ coincides with the Fubini-Study K\"ahler form $\omega_{FS}= i \partial \bar{\partial} \log (1+u)$  when
  $u > \varrho_1^2$, we actually  have 
   $$\omega_{\mathbf{N}}= \omega_{FS}  + i {\mathbf{N}}^{-1}\partial \bar{\partial}f$$ 
 when $\varrho > \varrho_1$. 
 Now choose 
 some $\varrho_2 > \varrho_1$, and let $\phi (u)\geq 0$ be a second smooth   cut-off function which is 
  $\equiv 0$ for $u \leq \varrho_1^2$ and $\equiv 1$ for $u \geq \varrho_2^2$. 
 We can then  consider the $(1,1)$-form on $M$ which is defined by 
  $$\hat{\omega}_{\mathbf{N}} =  {\omega}_{\mathbf{N}} - i \mathbf{N}^{-1} \partial \bar{\partial} [ \phi (u ) f]$$
  in the asymptotic regions, and given by ${\omega}_{\mathbf{N}}$ on the compact set $\mathbf{K}$; 
  the fact that  $f$ has been taken to be  $\Gamma_j$-invariant  guarantees that this 
 coordinate  expression is $\Gamma_j$-invariant, and so descends to a well-defined form on each end. 
  However, we then have 
  $$\hat{\omega}_{\mathbf{N}} =  \omega_{FS} +  i  \mathbf{N}^{-1} \partial \bar{\partial}[ (1-\phi ) f]$$
  in the asymptotic regions $\varrho \geq \varrho_1$. Since the Hessian of $(1-\phi)f$ is uniformly bounded with respect to 
  the Fubini-Study metric 
  on the compact   union of the transition annuli $ \varrho_1 \leq \varrho \leq\varrho_2$, 
  it follows that $\hat{\omega}_{\mathbf{N}}$ will be  positive-definite on these annuli for  $\mathbf{N}\gg 0$. 
  On the other hand,  since $\hat{\omega}_{\mathbf{N}}$ agrees with either $\omega_{FS}$ or ${\omega}_{\mathbf{N}}$ everywhere else, 
 it follows that, provided  we take ${\mathbf{N}}$ to be sufficiently large, 
  $\hat{\omega}_{\mathbf{N}}$ will be a  K\"ahler form on all of $M$. 
 But the K\"ahler metric $\hat{g}$ corresponding to  $\hat{\omega}:=\hat{\omega}_{\mathbf{N}}$ for some such suitably  large $\mathbf{N}$
 then exactly coincides  with the standard Fubini-Study metric of each $\CP_m/\Gamma_j$ in the asymptotic region $\varrho > \varrho_2$ of each end, 
 and so naturally extends to all of $(X,J_X)$ as a K\"ahler metric. 
 This shows    that the complex orbifold $X$ is indeed of K\"ahler type,  as claimed. 
\end{proof}

\begin{rmk}
As was  pointed out to us by Ronan Conlon, the above result can be generalized to asymptotically conical K\"ahler manifolds, even when 
 the end is not rationally  $2$-connected. For details, see  \cite[Theorem A (iv)]{heinlon2}. 
\end{rmk}

This now implies  a previously promised result:

\begin{prop} \label{class2} If $m\geq 3$, 
an  ALE K\"ahler $m$-manifold  has only one end. 
\end{prop}
\begin{proof}
Let $(M^{2m},g,J)$ be an ALE K\"ahler manifold, where  $m\geq 3$. 
Consider the orbifold compactification $(X,J_X)$ of $(M,J)$ given by Lemma \ref{compactify}, and let $\hat{g}$ be an orbifold  K\"ahler metric on $X$, with 
K\"ahler form $\hat{\omega}$. We may then consider the intersection pairing 
\begin{eqnarray*}
H^{1,1}(X, \RR) \times H^{1,1}(X, \RR) &\stackrel{Q}{\longrightarrow}&\quad  \RR \\
(~ [\alpha ] ~, ~ [\beta ]~ )\quad \qquad &\longmapsto& \int_X \alpha \wedge \beta \wedge \hat{\omega}^{m-2}~.
\end{eqnarray*}
on $H^{1,1}(X):=H^{1,1}_{orb}(X)$. 
However, because Hodge theory is valid in the orbifold setting,  there is a Lefschetz decomposition
  $$H^{1,1}(X, \RR) = \RR [\hat{\omega}] \oplus P^{1,1}(X, \RR),$$
  where the primitive harmonic $(1,1)$-forms $P^{1,1}$ are pointwise orthogonal to the K\"ahler form $\hat{\omega}$. 
 This implies a generalization of the Hodge index theorem: the intersection form $Q$  is of {\em Lorentz type}.
  Indeed, the Hodge-Riemann bilinear relations \cite{GH} tell us that $Q$ is
   positive-definite on $\RR [\hat{\omega}]$, and negative-definite on $P^{1,1}(X, \RR)$. 
   
   Let us now define a closed non-negative $(1,1)$-form $\alpha_j$ on $X$ supported in the closure of the $j^{\rm th}$ end of $M$ 
   by setting  $\alpha_j= i\partial \bar{\partial}F$  in the $j^{\rm th}$ asymptotic region, where $F$ is the function  defined by \eqref{gadget},
and then  extending $\alpha_j$ across the hyperplane at infinity as the Fubini-Study form $\omega_{FS}$, while  setting $\alpha_j\equiv 0$
 outside the closure of the $j^{\rm th}$  end.  The semi-positivity of $\alpha_j$ 
   then guarantees that $Q( [\alpha_j], [\alpha_j]) > 0$ for each $j$. However,  $Q( [\alpha_j], [\alpha_k]) =0$ 
   if $j\neq k$, since the supports of $\alpha_j$ and $\alpha_k$ are then disjoint. If $M$ had two or more ends, 
    $Q$ would thus admit  two orthogonal positive directions. But  since the generalized Hodge index theorem
says that $Q$ is of Lorentz type,   this is  impossible. 
To avoid this  contradiction, we are    thus forced to  conclude that 
$M$ can  only have one end. 
\end{proof}

\begin{rmk} The classes $[\alpha_j]\in H^{1,1}(X)$ in the above proof are proportional to the Poincar\'e duals of the hypersurfaces  $\Sigma_j = \CP_{m-1}/\Gamma_j$ 
arising as the hyperplanes at infinity of  the various ends.
 The fact that $Q( [\alpha_j], [\alpha_k]) =0$ for $j\neq k$ reflects the fact that $\Sigma_j \cap \Sigma_k= \varnothing$, while  the fact that 
 $Q( [\alpha_j], [\alpha_j]) > 0$ reflects  the fact that the homological self-intersection of 
 $\Sigma_j$ is represented by a  positive multiple of  the complex sub-orbifold $\CP_{m-2}/\Gamma_j$. 
This geometrical idea is the link between the above argument and  our proof of  Proposition \ref{class1}  in the complex surface case.

Various other means for proving  Proposition \ref{class2} are also available.  
For example, a  proof directly based on the pseudo-convexity of the boundary
can be found in \cite{weber}. 
Alternatively, uniqueness of the end can  be deduced by  applying \cite[Theorem 6.3]{rossivec} to the  Remmert reduction \cite{grauertex} of $(M,J)$.
 \end{rmk}

With  Lemma \ref{standard} and Proposition \ref{class2} in hand,  
 Theorem \ref{gamma} becomes comparatively easy to prove in complex dimension $m\geq 3$. Here is the key step:

\begin{prop}
\label{epigenome}
Let $(M^{2m},g,J)$ be an  ALE  K\"ahler manifold, $m\geq 3$, satisfying both conditions (\ref{eins}) and (\ref{zwei}), as set forth  on  page 
 \pageref{conditions}. 
Then, in {\bf any} asymptotic coordinate system,  its mass  is  given by 
$${\zap m}(M,g) = \lim_{\varrho\to \infty} \frac{1}{2(2m-1) \pi^m} \int_{S_\varrho/\Gamma} \theta \wedge \omega^{m-1} $$
for any  $1$-form $\theta$ with $d\theta = \rho$ 
on the end $M_\infty$, where $\rho$ is the Ricci form of $g$. 
\end{prop}
\begin{proof} Taking the Bartnik-Chru\'{s}ciel coordinate invariance  of the mass  \cite{bartnik,admchrus}  as given,  we will begin   
 by first checking  that the assertion is true in a particular asymptotic coordinate system and for a particular choice of $\theta$. 

Because $g$ is K\"ahler, the  asymptotic  complex coordinates $(z^1, \ldots , z^m)$ of Lemma \ref{standard} are all harmonic, 
and the same therefore applies to the real coordinates $(x^1, \ldots , x^n)$  obtained  by 
taking their real and imaginary parts. Thus 
$$\boldsymbol{\Gamma}^\ell:= g^{jk} \boldsymbol{\Gamma}^\ell_{jk}=  \Delta x^\ell=0,$$
so that 
$$g^{jk}\left( g_{ji,k}-\frac{1}{2}g_{jk,i}\right) =0$$
and 
$$
 g^{jk}\left(g_{j\ell , k} -  g_{jk,\ell}\right) = - \frac{1}{2}  g^{jk}g_{jk,\ell} = - \left(\log \sqrt{\det g}\right)_{,\ell }
$$ 
in this asymptotic coordinate system. On the other hand, our fall-off conditions guarantee that 
$$
g^{jk}\left(g_{j\ell , k} -  g_{jk,\ell}\right)= 
[\delta^{jk}+ O(\varrho^{1-m-\varepsilon})]\left(g_{j\ell , k} -  g_{jk,\ell } \right)= g_{i\ell,i}- g_{ii,\ell} + O(\varrho^{1-2m-2\varepsilon}), 
$$
and that the Hodge star operators associated with $g$ and $\delta$ differ by $O(\varrho^{1-m-\varepsilon})$. 
Thus 
$$\int_{S_\varrho/\Gamma} \left[ g_{ij,i} -g_{ii,j}\right] \mathbf{n}^jd\mathfrak{a}_E = -\int_{S_\varrho/\Gamma} \star \, d \log \sqrt{\det g}  \quad + O(\varrho^{-2\epsilon})$$
in these coordinates, and the mass is therefore given by  
$${\zap m}(M,g)= -\lim_{\varrho\to \infty}   \frac{(m-1)!}{4(2m-1)\pi^m} 
\int_{S_\varrho/\Gamma} \star \, d \log \sqrt{\det g}~.$$
However, the K\"ahler condition allows us to rewrite the integrand as 
$$\star \, d\log \sqrt{\det g} = \left[ -i(\partial - \bar{\partial}) \log \frac{\omega^m}{|dz|^{2m}}\right]\wedge \frac{\omega^{m-1}}{(m-1)!}$$
and since 
$$d \left[\frac{i}{2} (\partial - \bar{\partial}) \log \frac{\omega^m}{|dz|^{2m}}\right] = -  i\partial \bar{\partial} 
\log \frac{\omega^m}{|dz|^{2m}} =  \rho$$ 
on our K\"ahler manifold, we therefore have 
$${\zap m} (M,g) = \lim_{\varrho \to \infty} \frac{1}{2(2m-1) \pi^m} \int_{S_\varrho/\Gamma} \theta \wedge \omega^{m-1}$$
for a {\em particular} $1$-form 
$$\theta = \frac{i}{2} (\partial - \bar{\partial}) \log \frac{\omega^m}{|dz|^{2m}}
$$ 
with $d\theta = \rho$ on the end  $M_\infty$. 

On the other hand, since
$b_1(M_\infty )=0$, the most general $1$-form $\tilde{\theta}$ on $M_\infty$ with $d\tilde{\theta} = \rho$ is given 
by $\tilde{\theta} = \theta + df$ for a function $f$. Choosing a different $\theta$ would 
thus change the integrand by an exact form, and  so leave the integral on each $S_\varrho/\Gamma$
completely  unchanged. 

Finally, the limit is independent of the asymptotic coordinate system.
Indeed, notice that 
$$d \left[ \theta \wedge \omega^{m-1}\right] = \rho \wedge \omega^{m-1} =  \frac{s}{2m} \omega^m = {\textstyle\frac{(m-1)!}{2}} s~d\mu.$$
Consequently, if $\mathscr{S}$ is a real hypersuface  in the region $E_\varrho$ exterior to $S_\varrho/\Gamma$
such that $S_\varrho/\Gamma$ and $\mathscr{S}$ are the boundary components of a bounded region
$\mathscr{V} \subset E_\varrho$, then 
$${\textstyle \frac{2}{(m-1)!} }\left| 
\int_\mathscr{S} \theta\wedge \omega^{m-1} - \int_{S_\varrho/\Gamma} \theta\wedge \omega^{m-1}
\right| = 
\left| \int_\mathscr{V} s~d\mu \right| 
\leq \int_\mathscr{V} |s|~d\mu \leq \int_{E_\varrho} |s|~d\mu ,$$
and the expression at the far right tends to zero as $\varrho \to \infty$, since, by hypothesis, 
  the scalar curvature
$s$ belongs to $L^1$. 
\end{proof}

\begin{rmk} 
If the metric $g$ is {\em  scalar-flat} K\"ahler, the form $\theta\wedge \omega^{m-1}$
is actually closed, so the integral becomes  independent of the radius $\varrho$, and the mass can 
 be calculated without the need for taking a limit. 
 
 When $m=2$, 
the above argument still works if one simply {\em assumes} that there is an asymptotic chart  
in which $J$ is standard and $g$ falls off as in Lemma \ref{standard}. 
While this assumption does  in fact hold for many interesting examples, it unfortunately fails for the general ALE K\"ahler surface. 
This complication will force us to  develop  a more flexible approach  in order to be able to  definitively treat the  complex-surface case. 
 \end{rmk}

We now provide some key conceptual underpinning for our mass formula.

\begin{lem}\label{legit} 
 Let $(M,g)$ be any ALE manifold of real dimension $n\geq 4$. Then 
the natural map $H^2_c(M) \to H^2_{dR}(M)$ from compactly supported cohomology to ordinary de Rham cohomology is an isomorphism. Consequently,  every element of 
$H^2(M)$ is represented by a unique $L^2$ harmonic  $2$-form. 
\end{lem}
\begin{proof} We can compactify $M$ as a manifold-with-boundary $\overline{M}$ by adding a copy of $S^{n-1}/\Gamma_j$ to each end. 
The natural map in question therefore fits into an exact sequence 
$$\cdots \to H^1_{dR} (\cup_i[S^{n-1}/\Gamma_i] )\to H^2_c  (M) \to H^2_{dR}  (M) \to 
H^2_{dR}  (\cup_i[S^{n-1}/\Gamma_i])\to \cdots $$
corresponding to the exact cohomology sequence of the pair $(\overline{M}, \partial M)$. 
 On the other hand, since
de Rham cohomology injects upon    passing to  a finite cover, we have 
$H^k_{dR} (S^{n-1}/\Gamma_i ) \subset H^k_{dR}(S^{n-1})= 0$ when $0<k< n-1$. It therefore follows that 
 $H^2_c  (M)\to H^2_{dR}  (M)$ is an isomorphism. Moreover, since $g$ is asymptotically conical, 
this in turn implies \cite{carpar,hahuma} that 
$H^2(M)$ can be identified  with  the space ${\mathcal H}^2_2(M,g)$ of 
$L^2$ harmonic $2$-forms on $(M,g)$. 
\end{proof} 

This  entitles us  to make  the following definition: 
\begin{defn} \label{nutshell} 
If $(M,g,J)$ is any ALE K\"ahler manifold, we will use 
$$\clubsuit : H^2_{dR}(M) \to H^2_c(M)$$
to denote the {\bf inverse} of the natural map $H^2_c(M)\to H^2_{dR}(M)$. 
\end{defn}

We are now ready to state and prove our mass formula.

\begin{thm} \label{gist}
  The mass of any  ALE K\"ahler manifold  $(M,g,J)$ of complex dimension $m\geq 3$ is 
given by the formula
$${\zap m}(M,g) = - \frac{\langle \clubsuit (c_1) , [\omega ]^{m-1}\rangle}{(2m-1)\pi^{m-1}} +
\frac{(m-1)!}{4(2m-1)\pi^m} \int_M s ~d\mu$$
where  $\langle~, ~\rangle$ is the duality pairing between $H^2_c(M)$ and 
$H^{2m-2}(M)$.
\end{thm}
\begin{proof} 
Choose some $1$-form  $\theta$ on  $M_\infty$
such that  $d\theta= \rho$, where  $\rho$ is  the Ricci form of $(M,g,J)$. 
Next, choose some asymptotic coordinate system on $M_\infty$,
and  
 temporarily let   $\mathfrak{r}$  denote  the corresponding coordinate radius on  $M_\infty$. 
Finally,  choose a smooth cut-off function
$f:M\to [0,1]$   which is $\equiv 0$ 
 on $M-M_\infty$ and $\equiv 1$ for $\mathfrak{r} \geq \rad$, 
where  $\rad$ is some fixed large real number. We then  set  $\psi := \rho - d(f \theta)$. Since $M$ only has one end $M_\infty$  by Proposition \ref{class2}, this  
$\psi$ is then a compactly supported closed $2$-form on $M$. Since  $\psi$ is moreover cohomologous to $\rho$, it therefore  represents 
$\clubsuit ([\rho ]) = 
2\pi \clubsuit (c_1)$ in compactly supported cohomology.

  For any $\varrho >\rad$, we now
   let  $M_\varrho \subset M$ be the compact manifold-with-boundary obtained  by removing 
$\mathfrak{r} > \varrho$ from $M$, so that  $\partial M_\varrho = S_\varrho/\Gamma$.
Since $$\rho \wedge \omega^{m-1}= \frac{s}{2m} \omega^m=\frac{(m-1)!}{2} s~d\mu_g ,$$
 we have 
$$\frac{(m-1)!}{2}\int_{M_\varrho}s~d\mu= \int_{M_\varrho} \rho \wedge \omega^{m-1} = \int_{M_\varrho}  [\psi + d (f\theta) ]\wedge \omega^{m-1} ~.
$$
It follows  that 
\begin{eqnarray*}
2\pi  \langle \clubsuit (c_1) ,  [\omega ]^{m-1}\rangle &=& \int_M \psi \wedge \omega^{m-1} =   \int_{M_\varrho} \psi \wedge \omega^{m-1}\\
 &=& -  \int_{M_\varrho} d(f\theta \wedge \omega^{m-1} )  + \frac{(m-1)!}{2}\int_{M_\varrho}s~d\mu
 \\&=&   - \int_{\partial {M_\varrho}} f\theta \wedge \omega^{m-1} +\frac{(m-1)!}{2}\int_{M_\varrho}s~d\mu
 \\&=& - \int_{S_\varrho/\Gamma} \theta \wedge \omega^{m-1} +\frac{(m-1)!}{2}\int_{M_\varrho}s~d\mu~.
\end{eqnarray*}
In other words,  
$$
\frac{1}{2(2m-1)\pi^m}\int_{S_\varrho/\Gamma} \theta \wedge \omega^{m-1} = -\frac{ \langle \clubsuit (c_1) ,  [\omega ]^{m-1}\rangle}{(2m-1)\pi^{m-1}} + \frac{(m-1)!}{4(2m-1)\pi^m}\int_{M_\varrho}s~d\mu.
$$
Taking the limit of both sides as $\varrho \to \infty$  therefore yields 
$$ {\zap m} (M, g) = 
-\frac{ \langle \clubsuit (c_1) ,  [\omega ]\rangle}{(2m-1)\pi^{m-1}} + \frac{(m-1)!}{4(2m-1)\pi^m}\int_{M}s~d\mu 
$$
by Proposition \ref{epigenome}. This proves the desired mass formula. 
\end{proof}

Specializing to the scalar-flat case, we  now obtain  the high-dimensional version of Theorem \ref{alpha}:

\begin{thm} \label{atlas} 
If $(M^{2m},g,J)$ is any 
 ALE {\em scalar-flat}  K\"ahler $m$-manifold, $m\geq 3$, its mass is 
given by 
$${\zap m}(M,g) = - \frac{\langle \clubsuit (c_1) , [\omega ]^{m-1}\rangle}{(2m-1)\pi^{m-1}} . $$
In particular, the mass is a topological invariant in this context, entirely
determined by the smooth manifold $M$,  the first Chern class of 
the complex structure  and the K\"ahler class  of the metric. 
\end{thm}
 
\bigskip 

We now conclude our discussion of the high-dimensional case by pointing out some other useful consequences of Lemma \ref{compactify}. 
\begin{lem} \label{vanishing}
The orbifold  $(X,J_X)$  of Lemma \ref{compactify} satisfies $H^1(X, {\mathcal O})=0$. 
\end{lem}
\begin{proof}
By Lemma \ref{nostalgia}, the orbifold $(X,J)$ contains an open set of the form $\mathscr{U}/\Gamma$, where 
$\mathscr{U}\subset \CP_m$ is a tubular neighborhood of a hyperplane $\CP_{m-1}$, and this tubular neighborhood then contains a (perhaps smaller) 
neighborhood $\check{\mathscr{U}}$ of $\CP_{m-1}$ which is the union of all the projective lines $\CP_1 \subset \mathscr{U}$. 
If $\alpha\in H^0 (X, \Omega^1)$ is a global holomorphic $1$-form on the orbifold $X$, we can restrict it to $\mathscr{U}/\Gamma$ and then 
pull it back to obtain a holomorphic $1$-form $\hat{\alpha}\in H^0(\mathscr{U}, \Omega^1)$. However, 
the cotangent bundle of $\CP_m$ restricted to a projective line is isomorphic to ${\mathcal O}(-2) \oplus {\mathcal O}(-1)\oplus \cdots \oplus {\mathcal O}(-1)$,
so any holomorphic $1$-form on $\mathscr{U}$ must vanish identically along any projective line $\CP_1\subset \mathscr{U}$. It follows that $\hat{\alpha}$ vanishes
identically on $\check{\mathscr{U}}$. Hence $\alpha\equiv 0$ on  a non-empty open set, and hence $\alpha \equiv 0$ on $X$
by the uniqueness of  analytic continuation. Thus  $H^0(X, \Omega^1)=0$. 
However, $H^0(X, \Omega^1)=H^{1,0}(X)$ and $H^1 (X, {\mathcal O}) = H^{0,1} (X)$ are conjugate by Hodge symmetry, since $X$ admits K\"ahler metrics. 
This shows that $H^1 (X, {\mathcal O})=0$, as claimed. 
\end{proof}

In the asymptotically Euclidean case, this now allows us to prove  a result that will  play a leading role in \S \ref{positive} below:  

\begin{prop} \label{ae2} When $m\geq 3$, 
any AE K\"ahler $m$-manifold $(M^{2m},g,J)$ admits a 
proper holomorphic degree-one map  $M\to \CC^m$. 
\end{prop}
\begin{proof}
 Since $\Gamma= \{ 1\}$ by assumption, 
the compactification $(X,J_X)$ of $(M,J)$ is a manifold, and by Lemma \ref{nostalgia} it contains an open set biholomorphic to some tubular neighborhood $\mathscr{U}$ of
$\CP_{m-1} \subset  \CP_m$. In particular, $X$ contains a complete,  $m$-complex-dimensional  family of hypersurfaces arising as hyperplanes
in $\mathscr{U}\subset \CP_m$. Since $H^1(X, {\mathcal O}) =0$, holomorphic line bundles on $X$ are classified by their Chern classes, 
and it therefore follows that all of these hypersurfaces determine the same divisor line bundle $L\to X$; that is, they all belong the 
same $m$-dimensional linear system $|H^0(X, {\mathcal O}(L))|$.  Since no point   belongs to all of these 
hypersurfaces, this linear system has no base locus, and  it therefore  gives rise to a globally defined  holomorphic map 
$$\Phi: X\rightarrow \PP [ H^0 ( X, {\mathcal O}(L))^*]\cong \CP_m.$$
Since the hyperplanes we initially considered lie entirely within $\mathscr{U}$ and give projective coordinates on some smaller tubular neighborhood $\check{\mathscr{U}}$ of $\CP_{m-1}$, 
this map  takes $\check{\mathscr{U}}$ biholomorphically to its image, and no point of $\Phi (\check{\mathscr{U}})$ has any other pre-image in $X$. 
Thus $\Phi$ has degree $1$, and  the hyperplane $\Sigma\cong \CP_{m-1}$ used to compactify $M$  is taken biholomorphically to a
hyperplane $\CP_{m-1}\subset \CP_m$. The  restriction  of $\Phi$ to $M=X-\Sigma$ therefore defines a proper, degree-one holomorphic map $M\to \CC^m$,
as desired. 
\end{proof}

\begin{rmk}
While the last result roughly says that any AE K\"ahler manifold is a generalized blow-up of $\CC^m$, it should   be emphasized
that  some rather complicated scenarios are in principle allowed when $m\geq 3$. For example, one could blow up $\CC^m$ at a point, then choose a smooth sub-variety
$V$ in the resulting exceptional $\CP_{m-1}$, then modify the blow-up $\tilde{\CC}^m$ by replacing  $V$ with its projectivized normal bundle, 
and then repeat this procedure. While Proposition \ref{ae1} below will provide  an analogous result when $m=2$, the low-dimensional picture  is  simpler, as 
 the most general AE K\"ahler surface will turn out to just be 
 an iterated blow-up of $\CC^2$ at isolated points. 
\end{rmk}

More  generally, the Kodaira-Baily embedding theorem \cite{baily}  and   Lemma \ref{vanishing}  together  imply that 
the  K\"ahler orbifold  compactification $X$ of $M$ given by  Lemma \ref{compactify}   is  always a complex projective variety; and  Lemma \ref{ado} below
leads to a similar result when $m=2$. After resolution of singularities, it is therefore easy to show that  any ALE K\"ahler manifold is biholomorphic to the complement of a 
rationally connected hypersurface in a   rationally connected smooth projective variety.

\section{Coordinate Invariance of the Mass}
\label{homerun}

Proposition  \ref{epigenome} shows that the mass of an ALE K\"ahler $m$-manifold of complex dimension  $m\geq 3$  can be calculated by
integrating a coordinate-independent differential form 
 over  a family of hypersurfaces that tends to infinity. 
This perhaps sounds like it should  imply  the K\"ahler case of 
the coordinate-invariance of the mass, in the sense of the celebrated results of Bartnik  \cite[Theorem 4.2]{bartnik} 
and Chru\'{s}ciel \cite[Theorem 2]{admchrus}. 
However,  our proof  actually proceeded by checking our asymptotic mass  formula 
in a {\em special} coordinate system, and then noticing that this  formula actually has
   an interpretation that  is essentially coordinate-free; to know that our  expression also coincides
with the standard expression for the mass in other charts,  we still had to rely on   Bartnik-Chru\'{s}ciel. 
In this section, 
we will remedy this   by proving  
 a more robust version of Proposition  \ref{epigenome} that  {\em directly} relates our integral to the standard mass expression in {\em any} asymptotic
chart in which the metric satisfies a weak fall-off hypothesis. One remarkable consequence of this argument will be  that {\em Proposition \ref{epigenome} still
holds when $m=2$}, even though Lemma \ref{standard} cannot be generalized to this setting. 
The following technical result is the linchpin of our argument:  

\begin{prop}
\label{keystep}
Let $g$ be a $C^2$ K\"ahler metric on  $(\RR^{2m}-\mathbf{D}^{2m})/\Gamma$, $m\geq 2$, where $\Gamma\subset \mathbf{SO}(2m)$ is some
finite group that acts without fixed-points on $S^{2m-1}$.
In the given coordinate system $(x^1, \ldots , x^{2m})$  on $\RR^{2m}-\mathbf{D}^{2m}$, 
assume that $g$ satisfies the weak fall-off hypothesis 
$$g_{jk}= \delta_{jk} +O(\varrho^{-\tau}), \qquad g_{jk,\ell}=  O(\varrho^{-\tau -1})$$
where  $\varrho=|x|$ and where $\tau = m-1+\varepsilon$ for some $\varepsilon > 0$. 
Then there is a continuously differentiable $1$-form $\theta$ on  $(\RR^{2m}-\mathbf{D}^{2m})/\Gamma$
such that 
$$
\int_{S_\varrho/\Gamma}\left[ g_{kj,k}-g_{kk,j}\right] \mathbf{n}^j d\mathfrak{a}_E  - \frac{2}{(m-1)!} \int_{S_\varrho/\Gamma} \theta \wedge \omega^{m-1} = O(\varrho^{-2\varepsilon})
$$
and such that $d\theta = \rho$, where $\rho$ is the Ricci form of $g$ with respect to a given compatible integrable almost-complex structure $J$. \end{prop}
\begin{proof} Let $J$ be a given almost-complex structure which is parallel with respect to $g$, and recall that this implies that $J$ is integrable. 
By the   argument used in the proof of Lemma \ref{ndn}, we may then find a unique constant-coefficient almost-complex structure $J_0$ on $\RR^{2m}$
such that 
$$J=J_0 + O(\varrho^{-\tau}), \qquad \triangledown J= O(\varrho^{-\tau -1}),$$
where $\triangledown$ denotes the Euclidean connection associated with the coordinate system and where $\varrho=|x|$. By rotating our coordinates if necessary, 
we may then assume that $J_0$ is the usual complex-structure tensor of $\CC^m$. Since $\Gamma$ preserves $J$ and acts by linear transformations, 
it automatically preserves $J_0$, too, so  we actually have $\Gamma \subset \mathbf{U}(m)$. 
We will now systematically work in the complex coordinates $(z^1, \ldots , z^m)$ associated with this picture of $J_0$. 

Per standard conventions \cite{bes}, we let $J$ act on 1-forms $\phi$ by $J\phi = -\phi\circ J$, thereby making it consistent with index-raising. 
With this understood,  then, at least at large radius,
$$J\, dz^\mu= -i( dz^\mu+ \mathscr{K}^\mu_{\bar{\nu}}d\bar{z}^{\bar{\nu}} + \mathscr{L}^\mu_\nu dz^\nu)$$
for a uniquely determined  collection of coefficients $\mathscr{K}^\mu_{\bar{\nu}}$ and $ \mathscr{L}^\mu_\nu$ with the same $C^1_{-\tau}$ fall-off as $J-J_0$.  
Since we consequently also have 
$$J\, d\bar{z}^{\bar{\mu}}= +i( d\bar{z}^{\bar{\mu}}+ \overline{\mathscr{K}^\mu_{\bar{\nu}}}d{z}^{{\nu}} + \overline{\mathscr{L}^\mu_\nu} d\bar{z}^{\bar{\nu}}),$$
applying $J$ again therefore gives us 
$$J^2dz^\mu \equiv   -\left( dz^\mu + 2\mathscr{L}^\mu_\nu dz^\nu\right) \quad \bmod  C^1_{-2\tau} .$$
The fact that $J^2=-I$ therefore implies that $\mathscr{L}^\mu_\nu \in C^1_{-2\tau}$, and hence that 
$$J\, dz^\mu\equiv -i( dz^\mu+ \mathscr{K}^\mu_{\bar{\nu}}d\bar{z}^{\bar{\nu}})  \quad \bmod C^1_{-2\tau} ,$$
thus allowing us to  sweep the $ \mathscr{L}^\mu_\nu$ into the error term in our calculations.

Now  consider the collection of $1$-forms defined by 
$$
\zeta^\mu:= {\textstyle \frac{1}{2}}\left( dz^\mu + i Jdz^\mu\right)  \equiv  dz^\mu + {\textstyle \frac{1}{2}}\mathscr{K}^\mu_{\bar{\nu}}d\bar{z}^{\bar{\nu}} \quad \bmod C^1_{-2\tau}. 
$$
These are all $(1,0)$-forms  with respect to $J$, so the $m$-form 
$$\varphi = \zeta^1 \wedge \cdots \wedge \zeta^m$$
is consequently of type $(m,0)$ with respect to $J$. 
If we now let 
$$\varphi_0 = dz^1 \wedge \cdots \wedge dz^m$$
denote the standard coordinate  $(m,0)$-form with respect to $J_0$,
then 
$$\varphi ~\equiv ~\varphi_0 ~-\, {\textstyle \frac{1}{2}}\sum_{\mu,\bar{\nu}=1}^m (-1)^\mu \mathscr{K}^\mu_{\bar{\nu}}\, d\bar{z}^{\bar{\nu}}
\wedge dz^1\wedge \cdots \wedge \widehat{dz^\mu} \wedge
\cdots  \wedge dz^m \quad \bmod C^1_{-2\tau}. $$
Consequently, 
\begin{equation}
\label{volfall}
\varphi\wedge \bar{\varphi} ~\equiv ~\varphi_0\wedge \bar{\varphi}_0 \qquad \bmod C^1_{-2\tau}, 
\end{equation}
even though we merely have  $\varphi- \varphi_0\in C^1_{-\tau}$. 

Because $g$ is K\"ahler with respect to $J$, we of course have $\nabla J=0$, where $\nabla$ denotes the Levi-Civita connection of $g$.
In our real asymptotic coordinate system $(x^1, \ldots, x^{2m})$ this statement takes the explicit form 
$$\triangledown_jJ^\ell_k  + \mathbf{\Gamma}_{je}^\ell J^e_{k} - \mathbf{\Gamma}_{jk}^eJ^\ell_e=0$$
where $\triangledown$ denotes the flat Euclidean coordinate connection, and where  $\mathbf{\Gamma}$ are once again the  Christoffel symbols of $g$. 
By thinking of $\mathbf{\Gamma}^\ell_{jk}$ as a matrix-valued $1$-form $[\mathbf{\Gamma}_j]$, 
we can now usefully rewrite this as 
$$\triangledown_j J = \Big[\,\, J \, , \, [\mathbf{\Gamma}_j] \,\,\Big]$$
or, equivalently, as 
$$\triangledown_j (J-J_0) = \Big[\,J_0\, , \,  [\mathbf{\Gamma}_j]\,\Big] + \Big[\,(J-J_0)\, , \,  [\mathbf{\Gamma}_j]\,\Big].$$
But since our fall-off conditions tells us that $\mathbf{\Gamma}= O(\varrho^{-\tau-1})$ and that 
$J- J_0 = O(\varrho^{-\tau})$,  we therefore have 
\begin{equation}
\label{parallax}
\triangledown_j (J-J_0) = \Big[\,J_0\, , \,  [\mathbf{\Gamma}_j]\,\Big] + O(\varrho^{-2\tau-1}).
\end{equation}
Now set 
$$\mathscr{K}:= \mathscr{K}^\mu_{\bar{\nu}} \frac{\partial}{\partial z^\mu}\otimes dz^{\bar{\nu}},$$
where the Einstein summation convention is understood, and notice that 
$$J\equiv J_0 + i \mathscr{K} - i \overline{\mathscr{K}} \qquad \bmod C^1_{-2\tau}.$$
Expressing the endomorphism $[\mathbf{\Gamma}_j]$ as
$$[\mathbf{\Gamma}_j] = \mathbf{\Gamma}^\mu_{j\nu} \frac{\partial}{\partial z^\mu} \otimes dz^\nu + \mathbf{\Gamma}^\mu_{j\bar{\nu}}\frac{\partial}{\partial z^\mu} \otimes dz^{\bar{\nu}} + \mathbf{\Gamma}^{\bar{\mu}}_{j\nu}\frac{\partial}{\partial z^{\bar{\mu}}} \otimes dz^\nu + \mathbf{\Gamma}^{\bar{\mu}}_{j\bar{\nu}}\frac{\partial}{\partial z^{\bar{\mu}}} \otimes dz^{\bar{\nu}}$$
we can thus rewrite \eqref{parallax} as  
$$i \mathscr{K}^\mu_{\bar{\nu},j} = 2i \mathbf{\Gamma}^\mu_{j\bar{\nu}} + O(\varrho^{-2\tau-1})$$
and so deduce that 
$$
\mathbf{\Gamma}^\mu_{j\bar{\nu}}= {\textstyle \frac{1}{2}}\mathscr{K}^\mu_{\bar{\nu},j} + O(\varrho^{-2\tau-1}).
$$ 
In particular, after decomposing the  index $j$ into parts of type $(1,0)$ and $(0,1)$ with respect to $J_0$, we consequently have 
\begin{equation}
\label{keystone}
\mathscr{K}^\mu_{\bar{\nu},{\lambda}} = 2\mathbf{\Gamma}^\mu_{{\lambda}\bar{\nu}}+ O(\varrho^{-2\tau-1}) 
\end{equation}
and
\begin{equation}
\label{chiral}
\mathscr{K}^\mu_{\bar{\nu},\bar{\lambda}} = 2\mathbf{\Gamma}^\mu_{\bar{\lambda}\bar{\nu}}+ O(\varrho^{-2\tau-1}). 
\end{equation}
Since  the Levi-Civita connection  $\nabla$ is torsion-free, equation  \eqref{chiral} then implies  that 
$$
\mathscr{K}^\mu_{\bar{\nu},\bar{\lambda}}-\mathscr{K}^\mu_{\bar{\lambda}, \bar{\nu}} = O(\varrho^{-2\tau-1}).
$$
Thus, with the Einstein summation convention understood,  
\begin{equation}
\label{jozu}
\mathscr{K}^\mu_{\bar{\nu}, \bar{\lambda}} d\bar{z}^{\bar{\nu}}\wedge  d\bar{z}^{\bar{\lambda}}=   O(\varrho^{-2\tau -1}), 
\end{equation}
in contrast to the $O(\varrho^{-\tau -1})$ fall-off we might have na{\"\i}vely expected.

The same sort of decomposition  also allows us to express the metric as 
$$g = g_{\mu\nu} dz^\mu\otimes dz^\nu+ g_{\mu\bar{\nu}}dz^\mu \otimes d\bar{z}^{\bar{\nu}} + g_{\bar{\mu}\nu} d\bar{z}^{\bar{\mu}} \otimes  dz^\nu
+g_{\bar{\mu}\bar{\nu}}d\bar{z}^{\bar{\mu}} \otimes d\bar{z}^{\bar{\nu}}$$
where symmetry and reality imply that
$$g_{\mu\nu}=g_{\nu\mu}=\overline{g_{\bar{\mu}\bar{\nu}}}=\overline{g_{\bar{\nu}\bar{\mu}}},
\qquad g_{\mu\bar{\nu}}= g_{\bar{\nu}\mu}=\overline{g_{\bar{\mu}\nu}} = \overline{g_{\nu\bar{\mu}}}.$$
Our fall-off hypothesis now becomes 
$$
g_{\mu\nu}\in C^1_{-\tau}, \qquad g_{\mu\bar{\nu}} -  \delta_{\mu\bar{\nu}} \in C^{1}_{-\tau},
$$
with the understanding  that $\delta$ denotes the standard Euclidean metric, so that 
$[\delta_{\mu\bar{\nu}}]$ is {\em one-half} times the identity matrix. 
The K\"ahler form is thus given by 
\begin{eqnarray*}
\omega&=& g(J\cdot , \cdot) \\
&\equiv&  i \, g_{\mu\nu} (dz^\mu+ \mathscr{K}^\mu_{\bar{\lambda}}dz^{\bar{\lambda}}) \otimes dz^\nu
+ i\, g_{\mu\bar{\nu}} (dz^\mu+ \mathscr{K}^\mu_{\bar{\lambda}}dz^{\bar{\lambda}}) \otimes d\bar{z}^{\bar{\nu}}\\
&&  - 
i\, g_{\bar{\mu}\nu} (d\bar{z}^{\bar{\mu}}+ \overline{\mathscr{K}^\mu_{\bar{\lambda}}}dz^{{\lambda}})  \otimes  dz^\nu
-i\, g_{\bar{\mu}\bar{\nu}} (d\bar{z}^{\bar{\mu}}+ \overline{\mathscr{K}^\mu_{\bar{\lambda}}}dz^{{\lambda}})  \otimes  d\bar{z}^{\bar{\nu}}\quad \bmod C^1_{-2\tau}\\
&\equiv& ig_{\mu\bar{\nu}}dz^\mu \wedge d\bar{z}^{\bar{\nu}} - 
{\textstyle \frac{i}{2}}  \delta_{\lambda [\bar{\mu}}  \mathscr{K}^\lambda_{\bar{\nu}]} d\bar{z}^{\bar{\mu}}\wedge d\bar{z}^{\bar{\nu}}
+  {\textstyle \frac{i}{2}}  \delta_{\bar{\lambda} [\mu}  \bar{\mathscr{K}}^{\bar{\lambda}}_{\nu ]}  d{z}^{\mu}\wedge d{z}^{\nu}
  \quad \bmod  C^1_{-2\tau}, 
\end{eqnarray*}
where we have used the fact that $\omega$ is anti-symmetric, and hence equal to its own skew part; here
    the square brackets, denoting skew-symmetrization, have simply been added for clarity. 
  Taking the $(0,2)$ component of $\omega$ with respect to $J_0$ thus gives us the interesting complex $2$-form
  $$\omega^{0,2} \equiv  -  {\textstyle \frac{i}{2}}  \delta_{\lambda [\bar{\mu}} 
   \mathscr{K}^\lambda_{\bar{\nu}]} d\bar{z}^{\bar{\mu}}\wedge d\bar{z}^{\bar{\nu}} \quad \bmod C^1_{-2\tau}.$$
Our fall-off conditions now imply that 
$$[d^*(\omega^{0,2})]_\ell = - g^{jk}\nabla_j(\omega^{0,2})_{k\ell} = - \delta^{jk}(\omega^{0,2})_{k\ell,j} + O (\varrho^{-2\tau -1}),$$
so that 
$$[d^*(\omega^{0,2})]_\lambda  =  O (\varrho^{-2\tau -1}),$$
while \eqref{keystone} implies that 
\begin{eqnarray*}
[d^*(\omega^{0,2})]_{\bar{\kappa}}&=& -\delta^{\mu\bar{\nu}} (\omega^{0,2})_{\bar{\nu}\bar{\kappa},\mu} + O (\varrho^{-2\tau -1})\\
&=& {\textstyle \frac{i}{2}}\delta^{\mu\bar{\nu}} \delta_{\lambda [\bar{\nu}}  \mathscr{K}^\lambda_{\bar{\kappa}], \mu}+ O (\varrho^{-2\tau -1})\\
&=& i \delta^{\mu\bar{\nu}} \delta_{\lambda [\bar{\nu}}  \mathbf{\Gamma}^\lambda_{\bar{\kappa}] \mu}+ O (\varrho^{-2\tau -1})\\
&=& 
{\textstyle \frac{i}{2}} \delta^{\mu\bar{\nu}} \left( g_{[\bar{\kappa} \bar{\nu}],\mu} +  g_{\mu [\bar{\nu}, \bar{\kappa} ]} - g_{\mu [ \bar{\kappa}, \bar{\nu} ]} \right) + O (\varrho^{-2\tau -1})\\
&=& 
-i  \delta^{\mu\bar{\nu}} g_{\mu [ \bar{\kappa}, \bar{\nu} ]}  + O (\varrho^{-2\tau -1}).\\
\end{eqnarray*}
Thus
$$d^* (\omega^{0,2}) =  -i\delta^{\mu\bar{\nu}}g_{\mu [\bar{\lambda}, \bar{\nu}]} d\bar{z}^{\bar{\lambda}}  + O(\varrho^{-2\tau -1}),$$
and complex conjugation then gives us 
$$d^* (\omega^{2,0}) =  i\delta^{\nu\bar{\mu}}g_{\bar{\mu} [{\lambda}, {\nu}]} d{z}^{{\lambda}} + O(\varrho^{-2\tau -1}).$$
  Setting 
$$\leo := i \omega^{2,0} - i \omega^{0,2},$$
we therefore  obtain a real co-exact $1$-form 
$$\gamma = d^* \leo = -\star d\star \leo$$
that is explicitly given  by  
$$ 
\gamma = - 2 \, \Re e\, ( \delta^{\mu\bar{\nu}}g_{\mu [\bar{\lambda}, \bar{\nu}]} d\bar{z}^{\bar{\lambda}}  ) + O(\varrho^{-2\tau -1}).
$$

Next, we consider the $1$-form  
$$\gimel= g_{jk} \mathbf{\Gamma}^j dx^k$$
obtained from the vector field $\mathbf{\Gamma}^j\frac{\partial}{\partial x^j}$ by index-lowering, where  
$$\mathbf{\Gamma}^j: = g^{k\ell}\mathbf{\Gamma}^j_{k\ell}= \Delta x^j .$$
The usual formula for the Christoffel symbol then tells us that 
\begin{eqnarray*}
\gimel_{\bar{\nu}}&=&g_{\kappa\bar{\nu}}\mathbf{\Gamma}^\kappa = \delta_{\kappa\bar{\nu}}\mathbf{\Gamma}^\kappa + O(\varrho^{-2\tau -1})\\
&=& \delta_{\kappa\bar{\nu}}g^{jk}\mathbf{\Gamma}^\kappa_{jk} + O(\varrho^{-2\tau -1})\\
&=& 2\delta_{\kappa\bar{\nu}}\delta^{\mu\bar{\lambda}}\mathbf{\Gamma}^\kappa_{\mu\bar{\lambda}}+ O(\varrho^{-2\tau -1})\\
&=& \delta^{\mu\bar{\lambda}} (g_{\mu \bar{\nu} , \bar{\lambda}} +g_{\bar{\lambda}\bar{\nu} , \mu} - g_{ \mu\bar{\lambda},\bar{\nu}}) + O(\varrho^{-2\tau -1})\\
&=& \delta^{\mu\bar{\lambda}}g_{\bar{\lambda}\bar{\nu} , \mu} + 2 \delta^{\mu\bar{\lambda}}  g_{\mu [\bar{\nu} , \bar{\lambda}]} + O(\varrho^{-2\tau -1}).
\end{eqnarray*}
On  the other hand, the trace of equation \eqref{keystone} tells us  that 
\begin{eqnarray*}
\mathscr{K}^\mu_{\bar{\nu}, \mu}&=&2 \mathbf{\Gamma}^\mu_{\bar{\nu}\mu } + O(\varrho^{-2\tau -1})\\
&=&g^{\mu \bar{\lambda}}(g_{\mu \bar{\lambda}, \bar{\nu}} + g_{\bar{\lambda} \bar{\nu}, \mu} - g_{\mu \bar{\nu}, \bar{\lambda}})  + O(\varrho^{-2\tau -1})\\
&=&\delta^{\mu \bar{\lambda}}(g_{\mu \bar{\lambda}, \bar{\nu}} + g_{\bar{\lambda} \bar{\nu}, \mu} - g_{\mu \bar{\nu}, \bar{\lambda}})  + O(\varrho^{-2\tau -1})\\
&=& \delta^{\mu\bar{\lambda}}g_{\bar{\lambda}\bar{\nu} , \mu} - 2 \delta^{\mu\bar{\lambda}}  g_{\mu [\bar{\nu} , \bar{\lambda}]} + O(\varrho^{-2\tau -1}).
\end{eqnarray*}
This gives us an identity
\begin{equation}
\label{simplicity}
\Re e\, (\mathscr{K}^\mu_{\bar{\nu}, \mu}d\bar{z}^{\bar{\nu}})  = {\textstyle \frac{1}{2}}\gimel  + 2\gamma
\end{equation}
which will eventually prove to be invaluable.

Now, because $g$ is K\"ahler, the Levi-Civita connection $\nabla$ of $g$  induces a connection on $K:=\Lambda^{m,0}$ by restriction,
and we will simply denote this connection by $\nabla$, too. Thus, in the asymptotic region where $\varphi\neq 0$ has been defined, 
$$\nabla \varphi =  \vartheta \otimes \varphi$$
for some complex-valued connection $1$-form $\vartheta$. Setting 
$$\alpha : = \vartheta^{0,1}, \qquad \beta := \vartheta^{1,0},$$
we will now  use \eqref{jozu} and the fact \cite{bes,huybr} that the Chern and Levi-Civita connections on $K=\Lambda^{m,0}$ coincide to 
compute $\alpha$ and $\beta$ modulo harmless error terms. 
Now since $\nabla$ is actually the Chern connection, $\nabla^{0,1}= \bar{\partial}:=\bar{\partial}_J$,
and hence 
\begin{eqnarray*}
\alpha \wedge \varphi &=& \nabla^{0,1}\varphi = \bar{\partial} \varphi =  d\varphi =  d(\varphi- \varphi_0)\\
  &=& -d\left[ {\textstyle \frac{1}{2}}\sum_{\mu,\bar{\nu}=1}^m (-1)^\mu \mathscr{K}^\mu_{\bar{\nu}}\, d\bar{z}^{\bar{\nu}}\wedge dz^1\wedge \cdots \wedge \widehat{dz^\mu} \wedge
\cdots  \wedge dz^m \right] + O(\varrho^{-2\tau-1})\\
 &=&-{\textstyle \frac{1}{2}}\sum_{\mu,\bar{\nu}=1}^m (-1)^\mu d\mathscr{K}^\mu_{\bar{\nu}}\wedge  d\bar{z}^{\bar{\nu}}\wedge dz^1\wedge \cdots \wedge \widehat{dz^\mu} \wedge
\cdots  \wedge dz^m  + O(\varrho^{-2\tau-1})\\
 &=&    {\textstyle \frac{1}{2}}\sum_{\mu,\bar{\nu},\bar{\kappa}=1}^m (-1)^\mu  \mathscr{K}^\mu_{\bar{\nu},\bar{\kappa}}    d\bar{z}^{\bar{\nu}}\wedge d\bar{z}^{\bar{\kappa}}\wedge
 dz^1\wedge \cdots \wedge \widehat{dz^\mu} \wedge
\cdots  \wedge dz^m  \\
&&-
{\textstyle \frac{1}{2}}\sum_{\mu,\bar{\nu}=1}^m \mathscr{K}^\mu_{\bar{\nu},\mu}  d\bar{z}^{\bar{\nu}}\wedge dz^1\wedge \cdots   \wedge dz^m + O(\varrho^{-2\tau-1}) \\
&=&
\left( -{\textstyle \frac{1}{2}}\sum_{\mu,\bar{\nu}=1}^m \mathscr{K}^\mu_{\bar{\nu},\mu}  d\bar{z}^{\bar{\nu}}\right) \wedge \varphi_0 + O(\varrho^{-2\tau-1}) 
\\
&=&
\left( -{\textstyle \frac{1}{2}}\sum_{\mu,\bar{\nu}=1}^m \mathscr{K}^\mu_{\bar{\nu},\mu}  \bar{\zeta}^{\bar{\nu}}\right) \wedge \varphi + O(\varrho^{-2\tau-1}), \\
\end{eqnarray*}
where we have used \eqref{jozu} to sweep $\mathscr{K}^\mu_{\bar{\nu},\bar{\kappa}}d\bar{z}^{\bar{\nu}}\wedge d\bar{z}^{\bar{\kappa}}$ into the error term. 
Since the $\bar{\zeta}^{\bar{\nu}}$ are a basis for $\Lambda^{0,1}_J$, this shows that 
\begin{eqnarray*}
\alpha &=&
 -{\textstyle \frac{1}{2}} \mathscr{K}^\mu_{\bar{\nu},\mu}  \bar{\zeta}^{\bar{\nu}}  + O(\varrho^{-2\tau-1}) 
\\&=& 
-{\textstyle \frac{1}{2}} \mathscr{K}^\mu_{\bar{\nu},\mu}  d\bar{z}^{\bar{\nu}} + O(\varrho^{-2\tau-1}),
\end{eqnarray*}
where the Einstein summation convention is of course understood. 

On the other hand, the Chern connection is also compatible with the Hermitian inner product $\langle ~,~\rangle$ induced on the canonical line bundle $K$ by $g$. 
Thus, if $h:= \|\varphi\|^2 = \langle \varphi , \varphi \rangle$, we have 
\begin{eqnarray*}
\partial h &=& \nabla^{1,0} \langle \varphi , \varphi \rangle\\
&=& \langle  \nabla^{1,0} \varphi , \varphi \rangle + \langle   \varphi , \nabla^{0,1}\varphi \rangle\\
&=& \langle  \beta \otimes \varphi , \varphi \rangle + \langle   \varphi , \alpha \otimes \varphi  \rangle\\
&=& \langle \varphi , \varphi \rangle \beta + \langle \varphi , \varphi \rangle \bar{\alpha}\\
&=& h (\beta + \bar{\alpha})
\end{eqnarray*}
and hence 
$$\beta = -\bar{\alpha} + \partial \log h.$$
On the other hand, 
$$h = \| \varphi\|^2 =  i^{m^2}m!  ~\frac{\varphi \wedge \bar{\varphi}}{\omega^m},$$
and equation \eqref{volfall}  therefore tells us
that 
\begin{eqnarray*}
\partial \log h &=& \partial \log \frac{\varphi \wedge \bar{\varphi}}{\omega^m}\\
&=& -\partial \log \frac{\omega^m}{\varphi_0 \wedge \bar{\varphi}_0}+\partial \log \frac{\varphi\wedge \bar{\varphi}}{\varphi_0 \wedge \bar{\varphi}_0} \\
&=&-\partial \log \sqrt{\det g}+  O(\varrho^{-2\tau-1}),
\end{eqnarray*}
since $\varphi_0\wedge \bar{\varphi}_0$ is just  a constant times the coordinate volume element $|dz|^{2m}$,
while $\omega^m$ is just a constant times the {metric} volume element of $g$. 
Thus, relative to the trivialization given by $\varphi$,
the connection form $\vartheta$ of the Chern connection on $K$ is given by  
$$
\vartheta= \alpha + \beta  = \alpha - \bar{\alpha}  - \partial \log \sqrt{\det g} + O (\varrho^{-2\tau -1})
$$
where $\partial:= \partial_J$,  and where 
$$
\alpha = -{\textstyle \frac{1}{2}} \mathscr{K}^\mu_{\bar{\nu},\mu}  d\bar{z}^{\bar{\nu}} + O(\varrho^{-2\tau-1}). 
$$

However,  the curvature of the Chern connection on $K$ is given by   $i\rho$, where $\rho$ once again denotes the  Ricci form of $(g,J)$. Thus 
$$ d \vartheta = i \rho,$$
and 
the real $1$-form $\theta$ defined by 
$$\theta = \Im m \, \vartheta = -{\textstyle \frac{i}{2}} (\vartheta - \bar{\vartheta})$$
therefore satisfies $d\theta = \rho$. In conjunction with \eqref{simplicity},  the above calculation thus shows that 
\begin{eqnarray*}
\theta &=& i (\bar{\alpha} - \alpha) + {\textstyle \frac{i}{2}} (\partial - \bar{\partial})  \log \sqrt{\det g} \qquad + O (\varrho^{-2\tau -1})\\
 &=& -J (\alpha+\bar{\alpha} ) + {\textstyle \frac{i}{2}} (\partial - \bar{\partial})  \log \sqrt{\det g} \qquad + O (\varrho^{-2\tau -1})\\
&=& J \, \Re e \, (\mathscr{K}^\mu_{\bar{\nu},\mu}  d\bar{z}^{\bar{\nu}} )   - {\textstyle \frac{1}{2}} J \, d  \log \sqrt{\det g} \quad + O (\varrho^{-2\tau -1})\\
&=& {\textstyle \frac{1}{2}} J \, (\gimel  +4\gamma  - \, d  \log \sqrt{\det g} )\quad + O (\varrho^{-2\tau -1})
\end{eqnarray*}
where the $1$-form $\gimel$ corresponds to $ \mathbf{\Gamma}^j:=g^{k\ell} \mathbf{\Gamma}^j_{k\ell}$ by index lowering,
and where the $1$-form $\gamma$ is co-exact. 
On the other hand, the K\"ahler condition also implies that any $1$-form $\phi$ satisfies
$$\star \phi = J\phi \wedge \frac{\omega^{m-1}}{(m-1)!},$$
so it follows that 
$$\star ( \gimel +4 \gamma - d  \log \sqrt{\det g}) =  \frac{2}{(m-1)!} \, \theta \wedge \omega^{m-1} + O (\varrho^{-2\tau -1}).$$
Notice, moreover, that both sides are invariant under the isometric linear action of $\Gamma\subset \mathbf{U}(m)$. 
Setting $\tau = m-1+ \varepsilon$, where  $\varepsilon > 0$, we therefore have 
$$
\int_{S_\varrho/\Gamma} \star ( \gimel - d  \log \sqrt{\det g}) +4 \int_{S_\varrho/\Gamma}\star \gamma = \frac{2}{(m-1)!} \int_{S_\varrho/\Gamma} \theta \wedge \omega^{m-1} + O (\varrho^{-2\varepsilon}).
$$
However, since $\star \gamma = d \star \leo$, the second integral on the left vanishes by Stokes' theorem, and 
we therefore deduce that 
$$
\int_{S_\varrho/\Gamma} \star ( \gimel - d  \log \sqrt{\det g}) = \frac{2}{(m-1)!} \int_{S_\varrho/\Gamma} \theta \wedge \omega^{m-1} + O (\varrho^{-2\varepsilon}).
$$
But since $\gimel - d  \log \sqrt{\det g}= O(\varrho^{-\tau -1})$,  and since $\star$ differs from the Euclidean Hodge star by $O(\varrho^{-\tau})$,
this implies that 
$$
 \int_{S_\varrho/\Gamma}\left[ \gimel_j - (\log \sqrt{\det g})_{,j}\right]\mathbf{n}^j d{\mathfrak a}_E = \frac{2}{(m-1)!} \int_{S_\varrho/\Gamma} \theta \wedge \omega^{m-1} + 
 O (\varrho^{-2\varepsilon}).$$
However, 
$$\gimel_j = {\textstyle\frac{1}{2}}g^{k\ell} ( g_{jk,\ell}+ g_{j\ell, k}- g_{k\ell, j}) =  g_{jk,k}- {\textstyle\frac{1}{2}}g_{kk,j} +  O (\varrho^{-2\tau -1})$$
and 
$$
(\log \sqrt{\det g})_{,j}= {\textstyle\frac{1}{2}}g_{kk,j} +  O (\varrho^{-2\tau -1}),
$$
and we therefore have 
$$
\int_{S_\varrho/\Gamma}\left[ g_{kj,k}-g_{kk,j}\right] \mathbf{n}^j d\mathfrak{a}_E  - \frac{2}{(m-1)!} \int_{S_\varrho/\Gamma} \theta \wedge \omega^{m-1} = O(\varrho^{-2\varepsilon}), 
$$
as claimed. 
\end{proof}

This now  implies our coordinate-invariant reformulation of the mass:

\begin{thm}
\label{basso}
Let $(M^{2m},g,J)$ be a    ALE  K\"ahler manifold of any complex dimension $m\geq 2$. 
Suppose only that $g$ is a $C^2$ metric whose scalar curvature $s$ belongs to $L^1$, and that, in
{\em some} real asymptotic coordinate system $(x^1, \ldots, x^{2m})$ on a  given end $M_\infty$, the metric $g$ has fall-off
$$g_{jk} = \delta_{jk} + O (|x|^{1-m-\varepsilon}), \qquad g_{jk,\ell} =  O (|x|^{-m-\varepsilon})$$
for some  $\varepsilon > 0$. 
Then the mass at the given end, expressed as the limit of an integral computed in these coordinates, is well-defined, and satisfies
 $${\zap m}(M,g) = \lim_{\varrho\to \infty} \frac{1}{2(2m-1) \pi^m} \int_{S_\varrho/\Gamma} \theta \wedge \omega^{m-1} $$
for any  $1$-form $\theta$ with $d\theta = \rho$ 
on the end $M_\infty$, where $\rho$ is the Ricci form of $g$. 
Moreover, the mass, determined in this manner, is coordinate independent; computing it in any other asymptotic coordinate system 
in which the metric satisfies this weak fall-off hypothesis will produce exactly the same answer.
\end{thm}
\begin{proof} 
By Proposition \ref{keystep}, there is a {\em particular} $1$-form $\theta$ with $d\theta= \rho$ such that 
\begin{equation}
\label{homefree}
\lim_{\varrho\to \infty} {\textstyle \frac{(m-1)!}{4\pi^m (2m-1)}} \int_{S_\varrho/\Gamma}\left[ g_{kj,k}-g_{kk,j}\right] \mathbf{n}^j d\mathfrak{a}_E 
 =\lim_{\varrho\to \infty} \int_{S_\varrho/\Gamma} \frac{ \theta \wedge \omega^{m-1}}{2 \pi^m(2m-1)}
\end{equation}
provided either limit exists. 
The left-hand side of equation \eqref{homefree}    is of course  the  coordinate definition of the mass associated with the given asymptotic chart.
 On the other hand,  the last paragraph  of the proof of Proposition \ref{epigenome} shows that, when   $s\in L^1$, 
  the limit on the  right-hand side of  \eqref{homefree}  exists and actually coincides with the limit obtained 
 by instead performing   the relevant integrals  
 on  the level sets of an {\em arbitrary} exhaustion function for  $\widetilde{M}_\infty$. 
 But since $\Gamma$ is finite, $H^1_{dR}(M_\infty)= \Hom (\Gamma, \RR) =0$, and any other primitive $\tilde{\theta}$ for the Ricci form 
 can be expressed as $\tilde{\theta} = \theta + df$; thus,  choosing a different primitive  $\tilde{\theta}$ would just change the integrand 
by an exact form, and so leave 
 the right-hand-side of \eqref{homefree} unchanged. This shows that the right-hand limit is coordinate-independent. Consequently, 
  the limit on the left-hand side of \eqref{homefree} exists and is also 
 independent of the choice of coordinates, provided we restrict ourselves to asymptotic  charts in which $g$ satisfies  the above weak fall-off hypothesis. 
 \end{proof}

Note that we obtain something stronger if  $(M,g,J)$ is a {\em scalar-flat} K\"ahler
manifold. Indeed, when $s\equiv 0$, the differential form $\theta \wedge  \omega^{m-1}$ is {\em closed}, and the 
integral 
$\int_\mathscr{S} \theta \wedge  \omega^{m-1}$ then 
only depends on the homology class of the compact  hypersurface $\mathscr{S}\subset M_\infty$.
One can thus 
replace  the limit on the right-hand of \eqref{homefree} with the integral on a single hypersurface! 
This remarkable fact played a central role in the process of discovery that led to the 
present  results.

In order to extend our proof of Theorem \ref{gamma} to the $m=2$ case, we now  lack only one last  ingredient:  the fact
that an ALE K\"ahler surface can only have one end. In the next section, we will show that this is indeed true. In the process, we will also discover other  
interesting  and useful results governing 
 the complex-analytic behavior  of ALE K\"ahler surfaces.

\section{Complex Asymptotics: The Surface Case} 
\label{jiffy}

As we saw in Lemma \ref{standard},  the complex structure of any ALE K\"ahler manifold of 
 complex dimension  $m\geq 3$   is 
 {\em standard at infinity}, 
in the sense that the complement of a suitable compact set is biholomorphic to
 $(\CC^m -\mathbf{D}^{2m})/\Gamma$, where $\mathbf{D}^{2m}\subset \CC^m$ is the closed unit ball. 
 However, 
 concrete examples show \cite{gibhawk,hondale2}  that  this is not generally true  when $m=2$.
 Nonetheless, many of our  high-dimensional results  still  have workable analogs in the complex surface
 case. For example, here is an $m=2$ version of Lemma \ref{ndn}:

 \begin{lem} \label{ado}
 Let ${M}_\infty$  be an end of an ALE K\"ahler surface $(M^4,g,J)$, 
where we just assume that, in some asymptotic chart, 
 the    metric  either has fall-off
 $$g_{jk}- \delta_{jk} \in C^1_{-\tau}$$
 for some $\tau > 3/2$, or else that 
 $$g_{jk}- \delta_{jk}\in C^{2,\alpha}_{-\tau}$$
 for some $\tau > 1$. 
  Then there is a (non-compact) complex surface $\mathscr{X}$
 containing an embedded holomorphic curve $\Sigma\cong \CP_1$ with self-intersection $+1$, such that 
 the universal cover 
 $\widetilde{M}_\infty$ of $M_\infty$ is biholomorphic to $\mathscr{X}-\Sigma$. 
 \end{lem}
 \begin{proof}
 If $\tau > 3/2$, the proof of Lemma \ref{ndn}  goes through with only minor improvements. 
 Indeed, suppose $\tau \geq  1 + \varepsilon$ for some $\varepsilon \in (1/2,1)$. Then the almost-complex structure $J$ constructed by our previous method
 will still be  of H\"older
 class $C^{0,\varepsilon}$. Since we have assumed that  $\varepsilon > 1/2$, the Hill-Taylor version \cite{hiltay} of Newlander-Nirenberg thus  says that 
 $J$ is integrable, in the sense of the existence of  complex coordinate charts, iff
  its Nijenhuis tensor vanishes in the distributional sense. 
   However, our  $J$  belongs to $W^{1,p}\cap C^{0,\varepsilon}$ for any $p \in  (4, 2/(1-\varepsilon))$, and its 
   Nijenhuis tensor thus has components of class $L^p$. But since the Nijenhuis tensor
   of $J$ vanishes in the classical sense on $M_\infty= \mathscr{X}-\Sigma$, it therefore vanishes
  almost everywhere; and since its components belong to $L^p$, this means they also vanish as  distributions.  The Hill-Taylor 
   theorem then tells us that $(\mathscr{X}, J)$ can be covered with local complex coordinate charts, and  that these will be   
     at least $C^{1,\varepsilon}$ with respect to the original atlas. 
     
  However, when $\tau \in (1, 3/2]$, this argument  breaks down, and we instead  need to assume that $g_{jk}- \delta_{jk}\in C^{2,\alpha}_{-\tau}$
  in order to obtain the desired conclusion. We proceed by an argument  exactly parallel  to that given in  \cite[Section 3.2]{hhn}. The key idea is to first change coordinates
  on $\CC^2 - \mathbf{D}^4$ in such a manner that all the complex lines through the origin in $\CC^2$ become $J$-holomorphic curves. 
  The reason for doing this is that, when  passing  from $\CC^m$ to $\CP_m$ by inverting a coordinate,
    the worst loss of regularity in our previous construction 
   occurred in the radial directions. 
  Improving the radial behavior of $J$ by imposing this gauge choice will allow us to overcome this difficulty. 
  Indeed, assuming that $g-\delta \in C^{2,\alpha}_{-\tau}$, Picard iteration \cite[Section 3.2, Step 1]{hhn} allows one to 
  construct such a change of coordinates 
 $\Phi : \CC^2 - \mathbf{B} \to \CC^2 - \mathbf{D}^4$, outside  a sufficiently large ball $\mathbf{B}$,  such that the components of $\Phi- \mbox{id}$  belong to 
  $C^{2,\alpha}_{1-\tau}$. (One does gain control of an extra derivative only along radial complex lines, where the problem we are solving 
  is elliptic,  but, due to the lack of ellipticity  in the transverse directions, this is all that we can expect.) 
  Thus $(\Phi^*J)- J_0$ is now  of class $C^{1,\alpha}_{-\tau}$, and moreover vanishes in  radial complex directions. 
  Thus, our previous fall-off analysis  shows that $J$ induces a complex structure on a neighborhood of a projective line $\CP_1\subset \CP_2$   that 
  is actually of H\"older class $C^{1,\varepsilon}$, where $\varepsilon = \min ( \alpha, \tau -1)$. In particular, the
  resulting almost-complex structure has vanishing Nijenhuis tensor by continuity, and so, by standard versions \cite{nijenwoo,malgrange} of Newlander-Nirenberg, 
   becomes standard in complex charts  that are at least  $C^{2,\varepsilon}$
  with respect to our original atlas. 
   \end{proof}

We remark that this Lemma was first  discovered in the special case of {\em scalar-flat} 
K\"ahler metrics \cite{chenlebweb,lebmero,lebmask,lebpoonsym}, where, at the outset, 
 one can arrange for $g$ to have much  better fall-off, and where
the relevant complex surface $\mathscr{X}$ actually arises as a hypersurface in a twistor space. 
Lemma \ref{ado} thus allows us to generalize various proofs from the narrow world of scalar-flat K\"ahler surfaces  to the present, broader context. 
In particular, an argument used in \cite{lebmask} now yields an  analog of Proposition \ref{class2}: 

\begin{prop}\label{class1}
Any ALE K\"ahler surface $(M^4,g,J)$ has 
only  one end. 
\end{prop}
\begin{proof}
 Lemma \ref{ado} allows 
us to construct an orbifold compactification $X$ of $M$  by adding a quotient of 
$\CP_1$ to each end. 
After blowing up, this produces a smooth compactification $X$ of 
$M$ which is  a non-singular complex surface. Moreover, the closure of each end of 
$M$ contains smoothly immersed rational curves of positive normal bundle, 
and each such curve has positive self-intersection.  Grauert's criterion  therefore guarantees \cite{bpv} that $X$ is projective, and
so in particular is of K\"ahler type. The Hodge index theorem therefore tells us that the intersection form 
on $H^{1,1}(X,\RR)$ must be of Lorentz type. However, the curves  arising from two different ends
of $M$ would necessarily be disjoint, and therefore would be orthogonal with respect to the intersection form. 
Since this would contradict the Hodge index theorem if there were two or more ends, we are therefore
forced to conclude that $M$ can only have one end. 
\end{proof}

Similarly, an argument from \cite{lebmero,lebpoonsym}  proves the following:

\begin{prop} \label{ae1} 
Any AE  K\"ahler surface  is biholomorphic to an iterated blow-up of $\CC^2$.
\end{prop}
\begin{proof}
In the asymptotically Euclidean case,  the compactification $X$ is actually a complex manifold,  obtained by adding
a $\CP_1$  of self-intersection $+1$ to $M$. Grauert's criterion \cite{bpv} thus  implies that $X$ is projective, and 
in particular is K\"ahler. However, we must have $H^{1,0}(X)=0$, since the Kodaira deformations  \cite{kodsub} of
this   $\CP_1$
sweep out an open subset of $X$, and since $\Lambda^{1,0}X$ must be isomorphic to $\mathcal{O}(-2) \oplus \mathcal{O}(-1)$
on any one of these smoothly embedded copies of $\CP_1$. Hodge symmetry therefore tells us that $H^{0,1}(X)=0$, and
it therefore follows that all of these rational curves actually belong to the  same linear system. 
Since the intersection of all these rational curves is empty, this linear system has no base locus. 
Since Kodaira's theorem \cite{kodsub} moreover tells us that the dimension of this family is $2$, 
this linear system defines  a non-singular holomorphic map  $X\to \CP_2$ which sends a neighborhood
of   $\Sigma \subset X$ biholomorphically to a neighborhood of a projective line $\CP_1\subset \CP_2$. 
Since bimeromorphic maps between compact surfaces always factor into blow-ups and
blow-downs, it follows that $X$ is obtained from $\CP_2$ by blowing up points away from
this   $\CP_1$. Deleting  the line at infinity, we thus see that $M$ is simply a blow-up of $\CC^2$. 
\end{proof}

In particular, Proposition \ref{ae1} tells us that  the complex structure of  an AE K\"ahler surface is  always  standard at infinity,
just as it was in higher dimensions. 
 We emphasize, however,   that  the corresponding statement is generally {\em false}  for  ALE K\"ahler surfaces.
 Here it is perhaps worth emphasizing that the proof of Proposition \ref{ae1} is {\em global} in nature. 
 This should be contrasted with the {\em local} type of rigidity displayed by Lemma \ref{nostalgia}, the proof of which fails in a crucial respect when $m=2$. 
 
  Indeed, 
  if $\Sigma\subset \mathscr{X}$ 
 is an embedded $\CP_1$ of self-intersection $+1$ in a non-compact complex surface,  the Kodaira family $\mathscr{Y}$ of its deformations
 still carries a holomorphic projective structure, but now {\em any} holomorphic projective structure on a complex surface locally arises in this fashion
   \cite{hitproj,lebthes}. While 
the Weyl  curvature   always vanishes for a $2$-dimensional projective structure,  most such  structures are certainly 
not flat. Indeed, the  obstruction to projective flatness in dimension $m=2$ is actually 
 measured by the {\em projective Cotton 
tensor}, which can locally be expressed as 
$C_{\lambda\mu\nu}= \nabla_{[\mu}r_{\nu]\lambda},$ 
where $\nabla$ is 
any torsion-free  holomorphic connection  that both represents the projective structure and induces a flat connection
on the canonical line bundle of  $\mathscr{Y}$,  and where $r$ denotes the Ricci tensor of $\nabla$. 
The Cotton tensor of $\mathscr{Y}$ at the base-point $o$ is actually the obstruction 
to the triviality of the {\em fourth} infinitesimal neighborhood $\Sigma^{(4)}$.  Nonetheless,  one can still prove the following:

 \begin{lem} \label{aria} 
 Let $\mathscr{X}$ be a (possibly non-compact) 
 complex surface, and let $\Sigma\subset \mathscr{X}$ be an embedded  $\CP_1$ of self-intersection $+1$. Then the third infinitesimal neighborhood
 $\Sigma^{(3)}$
of $\Sigma\subset \mathscr{X}$ is isomorphic to the third infinitesimal neighborhood of a projective line $\CP_1 \subset \CP_2$. 
 \end{lem} 

Combining this with  the proof of Proposition \ref{standard} then yields 

 \begin{prop} \label{oscar}
Let $(M^4,g,J)$ be an ALE  K\"ahler surface. Then there is an  asymptotic  coordinate system $(x^1, \ldots, x^4)$ on the universal cover of the end $M_\infty$ of
$M$ in which 
$$g = \delta + O(|x|^{-1-\varepsilon}), \qquad \triangledown g = O(|x|^{-2-\varepsilon})$$
and 
$$
J = J_0 + O(|x|^{-3}), \qquad \triangledown J= O (|x|^{-4})
$$
where $\triangledown$ is the coordinate (Euclidean) derivative, and $J_0$ is the familiar constant-coefficient almost complex structure tensor 
on  $\CC^2=\RR^4$. 
\end{prop}

Thus, while one cannot always arrange for $J$ to be standard at infinity, it is at least asymptotic to the standard complex structure to a higher order
than the fall-off of the metric would na\"ively lead one to expect. In the asymptotic coordinates provided by Proposition \ref{oscar}, the proof of
Proposition \ref{keystep} then simplifies dramatically, because the $1$-forms $\gimel$ and $\gamma$  become negligible  error terms. Assuming
the Bartnik-Chru\'{s}ciel coordinate-invariance of the mass, 
 a variant  of the demonstration of Proposition \ref{epigenome}
 thus suffices to prove the $m=2$ case of the result. This was how we first obtained the   asymptotic mass formula in the complex-surface case.

While 
Lemma \ref{aria} cannot be improved in general,  one can still do systematically better in many cases of interest. Indeed,  notice  the action of   $\Gamma$ 
 on $\widetilde{M}_\infty$ always extends to a holomorphic action on $\mathscr{X}$, and that this then induces an action on $\mathscr{Y}$ preserving
 both the holomorphic projective structure and the base-point $o\in \mathscr{Y}$. Moreover, the induced action of $\Gamma$  on $T^{1,0}\mathscr{Y}\cong \CC^2$ 
 is just given by the tautological $2$-dimensional representation of $\Gamma \subset \mathbf{U}(2)$. 
Since the Cotton tensor at $o$ must be invariant under the action of $\Gamma$, it either vanishes, or else the action of $\Gamma$ on 
$[\CC^2 \otimes \Lambda^2(\CC^2)]^*$ must have a trivial $1$-dimensional sub-representation.
In our context, this will force $\Sigma^{(4)}$ to be standard, allowing one to  osculate $J$ by $J_0$   to higher order at infinity,   unless 
$\Gamma$ is a cyclic group $\ZZ_\ell$, where  $\ell$ is odd, acting on $\CC^2$ with generator
$$\left[\begin{array}{cc}e^{2\pi i/\ell} & 0 \\0 & e^{-4\pi i/\ell}\end{array}\right]. $$
However,  the  $\ell=3$ examples of Honda \cite{hondale2} show that   Proposition \ref{oscar} is actually optimal for certain ALE scalar-flat K\"ahler surfaces.

\section{The Mass Formula for  Complex Surfaces} 

All the pieces   needed to  finish the proofs of Theorem \ref{alpha} and \ref{gamma} are now in place.
Of course, the remaining step is  to   demonstrate the $m=2$ case of
the mass formula. Once this is done,  we will then  obtain Theorem \ref{beta} by simply re-examining   some off-the-shelf  examples
using these new instruments.

\begin{thm} \label{aupair}
The mass of any  ALE K\"ahler surface $(M^4,g,J)$ is given by 
$$
{\zap m}(M,g) = -\frac{1}{3\pi} \langle \clubsuit (c_1) ,  [\omega ]\rangle + \frac{1}{12\pi^2}\int_M s_g~d\mu_g~.
$$
\end{thm}
\begin{proof} 
Theorem  \ref{basso} shows that the asymptotic  mass formula of  Proposition \ref{epigenome}
also holds in the $m=2$ case. Meanwhile, Proposition \ref{class1}  shows that the $m=2$ version  of Proposition \ref{class2} 
also holds. 
With these minor substitutions,   the proof
 of Theorem \ref{gist}, with $m$ set now equal to $2$,  then proves the desired cohomological mass formula.
\end{proof}

In conjunction with Theorem  \ref{gist}, Theorem  \ref{aupair} now  implies   Theorem \ref{gamma}. 
We also obtain the following  corollary:

\begin{thm} \label{atlast}
The mass of any ALE {\em scalar-flat} 
K\"ahler surface $(M^4,g,J)$ is given by 
$$
{\zap m}(M,g) = -\frac{1}{3\pi} \langle \clubsuit (c_1) ,  [\omega ]\rangle .
$$
In particular,  the mass is a topological invariant in this setting, 
and depends only on the underlying  
manifold $M$, together with the cohomology classes $c_1(M,J)$ and $[\omega]$. 
\end{thm}

Theorem \ref{alpha}   is now an immediate consequence  of  Theorems \ref{atlas} and \ref{atlast}.

\bigskip

Let us now recast Theorem \ref{atlast} in a  more concrete form by identifying $H^2_c(M, \RR)$ with the homology group $H_2(M, \RR)$ via Poincar\'e duality. 
In this setting, 
the intersection pairing on $H^2_c(M)$ becomes the geometric pairing
on $H_2(M)$ 
obtained by counting intersection numbers of compact  (real)  surfaces in $M$. Note that  Lemma \ref{legit} implies that  this pairing 
$$H_2(M, \RR)\times H_2(M, \RR)\to \RR$$ is non-degenerate on any ALE $4$-manifold $M$.

\begin{thm}\label{onpar} 
Let $(M,g,J)$ be an ALE scalar-flat K\"ahler surface. Let $E_1, \ldots E_\mathfrak{b}$ be a basis for $H_2(M, \RR )$,
and let $Q = [Q_{jk}]= [E_j\cdot E_k ]$ be the corresponding intersection matrix.  If we define  $a_1 , \ldots , a_\mathfrak{b}$ by 
\begin{equation}
\label{caementium}
\left[\begin{array}{c}a_1 \\ \vdots  \\a_\mathfrak{b}\end{array}\right]=  \quad \raisebox{-.07in}{\mbox{\Huge $Q^{-1}$}}\left[\begin{array}{c}\int_{E_1} c_1  \\  \vdots   \\ \int_{E_\mathfrak{b}} c_1 \end{array}\right]
\end{equation}
then the mass of $(M,g)$ is given by 
\begin{equation}
\label{cement}
{\zap m}= -\frac{1}{3\pi} \sum_{j=1}^\mathfrak{b} ~a_j \int_{E_j} [\omega ]
\end{equation}
where $[\omega]$ denotes the K\"ahler class of $(M,g,J)$. 
\end{thm}

\begin{proof}
The cycle $\sum a_jE_j$ is exactly determined by the requirement that 
$$\left( \sum a_jE_j\right)\cdot D = \int_D c_1$$
for any $D\in H_2(M)$. However, since $H_2(M, \RR) = H^2_c(M)$ by Poincar\'e duality, this is equivalent to saying that 
$$\sum a_j \int_{E_j} \mho = \langle  c_1, \mho \rangle$$
for any $\mho \in H^2_c(M)$. However, for any $\Omega\in H^2(M)$, we  have
$$ \langle  \clubsuit (c_1), \Omega \rangle =  \langle  \clubsuit (c_1), \clubsuit (\Omega ) \rangle = \langle  c_1, \clubsuit (\Omega ) \rangle $$
so that, setting $\mho= \clubsuit (\Omega )$, we have  
$$\langle  \clubsuit (c_1), \Omega \rangle = \sum a_j \int_{E_j} \clubsuit (\Omega )= \sum a_j \int_{E_j} \Omega~.$$
Setting $\Omega = [\omega ]$, we therefore have 
$${\zap m}(M,g,J) = -\frac{1}{3\pi} \langle  \clubsuit (c_1), [\omega ] \rangle = -\frac{1}{3\pi}\sum a_j \int_{E_j} [\omega ]$$
 by Theorem \ref{aupair}. \end{proof}
 
 Recalling Proposition \ref{ae1}, we thus obtain the following:  
 
 \begin{cor} \label{safe}
 Let $(M,g,J)$ be an AE scalar-flat K\"ahler surface.  We may then
 choose  a homology basis $E_1, \ldots , E_\mathfrak{b}\in H_2(M, \ZZ)$ with 
 intersection matrix $Q=-I$  in which $c_1(M)$ is Poincar\'e dual to  $-\sum E_j$. Consequently, 
 $${\zap m}(M,g)= \frac{1}{3\pi} \sum_{j=1}^\mathfrak{b} \int_{E_j} [\omega]$$
 where $[\omega]$ is the K\"ahler class of $(M,g,J)$. 
 \end{cor}
 \begin{proof}
 By  Proposition \ref{ae1},  $(M,J)$ is 
 an iterated blow-up of $\CC^2$ at $\mathfrak{b}$ points, and so  has a small deformation which is a blow-up of $\CC^2$ at distinct points. 
 One can then take the $E_j$ to be the homology classes of the exceptional divisors of these distinct points. 
 \end{proof}
    
 When $\CC^2$ is  blown up at   distinct  points,  
 the expression for the mass provided by Corollary \ref{safe}  is obviously a sum of areas of holomorphic curves, and so is certainly positive
 if $\mathfrak{b} > 0$. However, $\sum_{j=1}^\mathfrak{b} E_j$ is {\em always} homologous to 
 a sum of holomorphic curves with positive integer coefficients, even 
  in the degenerate cases, so this expression for the mass 
 will  actually always be positive whenever 
 $M\neq \CC^2$. We will return to this point in Theorem \ref{roger} below.

 \begin{cor}\label{generic} Let $(M,g,J)$ be an ALE scalar-flat K\"ahler surface, where $(M,J)$ is obtained from 
 the total space of the ${\mathcal O}(-\ell)$ line bundle over $\CP_1$ by  blowing up $\mathfrak{b}-1$
distinct  points that do not lie on the zero section. Let $F$ be  the homology class of the zero section,  and let $E_1, \ldots , E_{\mathfrak{b}-1}$ be the homology classes of the
exceptional divisors of the blown up points. 
 Then
$$
{\zap m}(M,g) = \frac{1}{3\pi} \left[ \frac{2-\ell}{\ell} \int_F\omega + \sum_{j=1}^{\mathfrak{b}-1} \int_{E_j}\omega \right]~.
$$
 \end{cor}
\begin{proof} In the homology basis $F, E_1, \ldots , E_{\mathfrak{b}-1}$, the intersection form is represented by the matrix 
$$
\raisebox{-.07in}{\mbox{\Huge Q}}\quad =~ \left[\begin{array}{cccc}-\ell &  &  &  \\ & -1 &  &  \\&  & \ddots &  \\ &  &  & -1\end{array}\right]
$$
while 
$$
\left[\begin{array}{c}\int_F c_1 \\ \int_{E_1}c_1 \\ \vdots \\ \int_{E_{\mathfrak{b}-1}}c_1\end{array}\right]=\left[\begin{array}{c}2-\ell \\1 \\ \vdots \\1\end{array}\right]
$$
and the result therefore follows from Theorem \ref{onpar}. 
\end{proof}

In particular, one sees that the mass is negative when $\ell \geq 3$ and no points are blown up. 
This was laboriously discovered by hand for specific  explicit examples in \cite{lpa},
but now we see that this phenomenon occurs as a matter of general principle. 

Of course, the mass formula we have discovered is purely topological, and  thus
insensitive to deformations of complex structure. As an application, we  immediately now see the following:

 \begin{cor}\label{gungho} 
 Let $(M,g,J)$ be an ALE scalar-flat K\"ahler surface, where $(M,J)$ is obtained from 
 the total space of the ${\mathcal O}(-\ell)$ line bundle over $\CP_1$ by  blowing up $\mathfrak{b}-1$
distinct  points that lie on the zero section. Let $\tilde{F}$ be  the homology class of the {proper transform}
of the zero section,  and let $E_1, \ldots , E_{\mathfrak{b}-1}$ be the homology classes of the
exceptional divisors of the blown up points. 
 Then
$$
{\zap m}(M,g) = \frac{1}{3\pi \ell} \left[ (2-\ell) \int_{\tilde{F}}\omega + 2 \sum_{j=1}^{\mathfrak{b}-1} \int_{E_j}\omega \right]~.
$$
 \end{cor}
\begin{proof} This example is diffeomorphic to the previous one, 
in a manner that preserves the first Chern class. 
The mass formula therefore follows from Corollary \ref{generic}, together with the observation that 
$\tilde{F} + E_1+ \cdots +E_{\mathfrak{b}-1}$   is homologous to $F$.
\end{proof}

Applying  Corollary \ref{gungho}
to some examples constructed in \cite{mcp2},  we now immediately 
obtain Theorem \ref{beta}: 

\begin{thm} There are infinitely many topological types of ALE scalar-flat K\"ahler surfaces 
that have zero mass, but are not Ricci-flat. Indeed, for any $\ell \geq 3$, the blow-up 
of the $\mathcal{O}(-\ell)$ line bundle on $\CP_1$ at any  non-empty collection of  distinct points
on the zero section admits such metrics. 
\end{thm}
\begin{proof} 
Let  $p_0, p_1, \ldots , p_{\mathfrak{b}-1}$ be distinct points in hyperbolic $3$-space $\mathcal{H}^3$, 
chosen so that  that the geodesic rays $\overrightarrow{p_0p_1}, \ldots , \overrightarrow{p_0p_{\mathfrak{b}-1}}$
all have distinct initial tangent directions at $p_0$. Let $\mathfrak{r}_j$, $j=0, \ldots , \mathfrak{b}-1$, denote the hyperbolic distance
from $p_j$, considered as a function on $\mathcal{H}^3$. Let  $X={\mathcal H}^3-\{ p_0, p_1, \ldots , p_{\mathfrak{b}-1}\}$,
and let $P\to X$ be the principal $\mathbf{U}(1)$-bundle with $c_1=-\ell$
on a small $2$-sphere around  $p_0$ and $c_1=-1$ on a small $2$-sphere around any other $p_j$. 
Set 
$$V := 1+ \frac{\ell}{e^{2\mathfrak{r}_0}-1}+ \sum_{j=1}^{\mathfrak{b}-1}\frac{1}{e^{2\mathfrak{r}_j}-1},$$
on $X$, and let $\vartheta$ be a connection $1$-form on $P\to X$ with curvature  
$$d\vartheta = \star dV,$$
where the Hodge star is computed with respect to the hyperbolic metric ${\zap h}$ on $X\subset \mathcal{H}^3$ and standard orientation. 
Finally, let 
$${\zap g}= \frac{1}{4\sinh^2 \mathfrak{r}_{0}}\left[ V{\zap h}+ V^{-1}\vartheta^2\right]$$
on $P$, and let $(M,g)$ be the metric completion of $(P,{\zap g})$. Then $(M,g)$ is an ALE 
scalar-flat K\"ahler surface, and $(M,J)$ is obtained \cite[p. 244]{mcp2} from the $\mathcal{O}(-\ell)$ line bundle on $\CP_1$
by blowing up $\mathfrak{b}-1$ distinct points on the zero section. The proper transform $\tilde{F}$ of the zero section is 
represented in this picture by the sphere at infinity of $\mathcal{H}^3$, and the restriction of $g$ to 
$\tilde{F}$ is just the standard Fubini-Study metric, with total area $\pi$. 
On the other hand, the exceptional curve $E_j$ is the closure in $M$ of the inverse image in $P$ of the geodesic ray in $\mathcal{H}^3$ which starts at $p_j$ and
points diametrically away from $p_0$; its total area is given by $2\pi/(e^{2\rad_j}-1)$, where $\rad_j= \mathfrak{r}_0(p_j)$ is the hyperbolic distance
from $p_0$ to $p_j$. By Corollary \ref{gungho}, the mass of the resulting metric is therefore given by 
$${\zap m}(M,g)= \frac{1}{3\ell}\left[
2-\ell + 4\sum_{j=1}^{\mathfrak{b}-1} \frac{1}{e^{2\rad_j}-1}
\right]$$
and so, if $\ell\geq 3$ and $\mathfrak{b}-1\geq 1$, this obviously changes sign as we let the $\rad_j$ range over all of $\RR^+$. To be more concrete and specific, we in particular
  obtain a non-Ricci-flat example with zero mass if 
we take $\mathfrak{b}-1=\ell-2\geq 1$ and $\rad_j=\log \sqrt{5}$ for every $j=1, \ldots , \mathfrak{b}-1$. 
\end{proof}

Interestingly,  though,  the above  construction  depends in practice on a choice of  $(M,J)$ which is 
{\em non-minimal}, in the sense of being the blow-up of another  complex surface. This appears 
to be essential. Indeed,  the following consequence of Theorem \ref{onpar}, 
 which was graciously  pointed out to 
us by Cristiano Spotti, offers a  systematic  
result  
along these lines:

\begin{cor} \label{thespot}
Let $(M^4,g,J)$ be an ALE scalar-flat  K\"ahler surface, and suppose that    $(M,J)$  is 
 the {\bf minimal}
resolution of a surface singularity. Then ${\zap m}(M,g)\leq 0$, with equality iff $g$ is Ricci-flat. 
\end{cor}
\begin{proof}
Choose a  basis for $H_2$ that is 
represented  by 
  a collection of smooth 
  rational curves $E_j$. Because the resolution is assumed to be minimal, each $E_j$ has
  self-intersection $\leq -2$, and  adjunction therefore tells us  that 
$\int_{E_j}  c_1\leq 0$  for every $j$. However,  it is also known 
\cite[Remark 3.1.2]{alexeev} that every entry in the inverse $Q^{-1}$ of the intersection matrix 
of such a minimal resolution 
is  {\em non-positive}. Thus, the coefficients  defined by equation \eqref{caementium}  all satisfy $a_j\geq 0$,
and the mass formula \eqref{cement} therefore produces a  non-positive answer. Moreover, 
if the mass is zero, then  $a_j=0$ for all $j$, so that $Q\vec{a}$ vanishes and $c_1=0$.
But Lemma \ref{legit} tells us that  $c_1$ is represented by a unique $L^2$ harmonic $2$-form,
and, since $g$ is scalar-flat K\"ahler, one such representative is $\rho/2\pi$, where $\rho$ is the Ricci-form of $(M,g,J)$. 
The mass  therefore  vanishes  for such a manifold
 if and only if the metric is Ricci-flat. 
\end{proof}

Lock and Viaclovsky  \cite{vialock} have recently given  a systematic construction of ALE scalar-flat K\"ahler metrics on  minimal resolutions of surface singularities,
thereby putting the earlier examples of Calderbank and Singer \cite{caldsing} into a broader context.
The above Corollary now shows  that all of these  examples actually  have  negative mass. 
 
\section{The Positive   Mass Theorem} 
\label{positive}

We   conclude this article by proving the positive mass theorem for K\"ahler manifolds, along
with our related Penrose-type  inequality. 

Suppose that $(M^{2m},g, J)$ is an AE K\"ahler manifold. Then, as we saw in Proposition \ref{ae2}, 
there is a proper holomorphic map $F: M\to \CC^m$ which  has degree $1$, and which is a biholomorphism outside a compact set.  We now consider the
holomorphic $m$-form 
$$\Upsilon = F^* dz^1\wedge \cdots \wedge dz^m ,$$
which is a holomorphic section of the canonical line bundle of $M$, and which exactly vanishes at the 
set of critical points of $F$. Because this zero set is locally the zero set of a non-trivial holomorphic function, 
it is purely of complex codimension $1$, and we moreover know this locus is 
compact because $F$ is a biholomorphism outside of a compact set. Breaking up the locus $\Upsilon=0$ as a finite union 
of  its irreducible components
$D_j$, and assigning each of these an integer multiplicity $n_j$ given by the order of
vanishing of $\Upsilon$ along $D_j$, we can thus express the divisor $D$  as 
$$D = \sum_j n_jD_j~.$$
Since $\Upsilon$ is a holomorphic section of the canonical line bundle $K_M$, 
the homology class $[D] = \sum n_j [D_j]$  is then 
Poincar\'e dual to $\clubsuit (c_1(K_M))= - \clubsuit (c_1(M,J))$. The mass formula 
of Theorem \ref{gamma} therefore can be rewritten as  
$${\zap m}(M,g) =  \frac{1}{(2m-1)\pi^{m-1}} [\omega]^{m-1} (D) +
\frac{(m-1)!}{4(2m-1)\pi^m} \int_M s_g d\mu_g ,$$
and we therefore obtain the  Penrose inequality promised by Theorem \ref{epsilon}:

\begin{thm}\label{roger}
Suppose that $(M^{2m},g, J)$ is an AE K\"ahler manifold with scalar curvature $s\geq 0$. Then, in terms
of the complex hypersurfaces $D_j$ and positive integer multiplicities $n_j$ described above, 
$${\zap m}(M,g) \geq   \frac{(m-1)!}{(2m-1)\pi^{m-1}}\sum_j n_j\mbox{Vol}\, (D_j)~,$$
with equality iff $(M,g,J)$ is scalar-flat K\"ahler. Moreover, $\bigcup_j D_j \neq  \varnothing$ if
$(M,J)\neq \CC^m$. 
\end{thm}
\begin{proof} Since we have assumed that $s\geq 0$, the scalar curvature integral in the mass formula is 
non-negative, and equals zero only if $g$ is scalar-flat. Since the volume form induced by $g$ 
on the regular locus of
$D_j$ is just $\omega^{m-1}/(m-1)!$, we can therefore transform $[\omega]^{m-1}(D)$ into 
$(m-1)!$ times a sum of volumes, weighted by multiplicities, and the stated inequality now follows
from the mass formula. 

Finally, $\bigcup_j D_j$ can only be empty if  $\Upsilon = F^*dz^1 \wedge \cdots \wedge dz^m$ is 
everywhere non-zero. But this happens iff $F$ has no critical points,  or in other words iff $F$ is a local 
diffeomorphism. However, $F$ is a degree $1$ proper holomorphic map. Thus the fact that $f$ is a
local diffeomorphism implies that 
it  is  actually a global biholomorphism.
\end{proof}

With this in hand, we can now easily read off our Positive Mass Theorem,  announced in the introduction as Theorem \ref{delta}: 

\begin{thm}
Suppose that $(M^{2m},g, J)$ is an AE K\"ahler manifold with scalar curvature $s\geq 0$. Then
its  mass ${\zap m}(M,g)$ is non-negative, and equals zero only if $(M,g,J)$ is flat.
\end{thm}
\begin{proof}
By Theorem \ref{roger}, the mass is positive unless $g$ is scalar-flat K\"ahler, 
$\bigcup_jD_j = \varnothing$,  and $(M,J)= \CC^m$. However, the Ricci form $\rho$ of 
an ALE scalar-flat K\"ahler
metric is an $L^2$ harmonic form, and so, by Lemma \ref{legit},  must vanish if the cohomology class $2\pi c_1$
it represents vanishes. Thus, $g$ would necessarily be a Ricci-flat AE metric on $\CC^m$.
But the AE condition implies that a metric's  volume growth is asymptotically exactly Euclidean, 
and the Bishop-Gromov inequality  thus implies that a complete Ricci-flat metric 
with this property is necessarily flat. 
\end{proof}

\appendix
\makeatletter
\def\@seccntformat#1{Appendix\ \csname the#1\endcsname :\quad}
\makeatother

\section{Normalization of the Mass}

In this appendix, we provide a  ``physical'' explanation
 of our normalization of the mass integral. We work throughout in units where $G=c=1$.
 
In the absence of matter,  tidal forces in Newtonian gravitation distort the shape of a cloud  of test particles without changing its volume, to lowest order in time. 
Thus, the acceleration vector field due to gravitation should  be divergence-free in empty space. If we assume that an isolated object generates
an acceleration field that points towards the object, with magnitude only  depending on the distance $\varrho$ from the source, 
 the   acceleration field in dimension $n$ must therefore take the form 
$${\vec{a}}= \nabla\left( \frac{{\zap M}}{\varrho^{n-2}}\right)$$
for some constant ${\zap M}$,  which we now declare to be the {\em mass} of the source. 
In the classical case of $n=3$, this of course reproduces Newton's law of gravitation. 
Since a test particle  following a circular orbit of radius $\varrho$ and  angular frequency $\omega$ about the origin exhibits  an in-pointing  radial 
acceleration of magnitude $\varrho \omega^2$, this acceleration can be ascribed to  our   gravitational  field iff 
$$\omega^2 = (n-2)\frac{{\zap M}}{\varrho^n}.$$
This  is a crude generalization of Kepler's third law of planetary motion.

Einstein's vacuum equations state that the Ricci curvature of the space-time metric 
should vanish in the absence of matter;  inspection of  Jacobi's equation  reveals that this is again equivalent 
to  requiring  that tidal forces  distort the shape of a cloud  of test particles without changing its volume, to lowest order in time. In space-time dimension 
$n+1$, the general spherically symmetric solution of these equations is the generalized Schwarzschild metric 
$${\zap g}=  -\left(1- \frac{A}{\varrho^{n-2}}\right) dt^2 +
\left(1- \frac{A}{\varrho^{n-2}}\right)^{-1} d\varrho^2 + \varrho^2  \mathfrak{h}$$
where $\mathfrak{h}$ denotes the standard unit-radius metric on $S^{n-1}$ and $A$ is a real constant. Notice
that 
$$\xi = \frac{\partial}{\partial t}  + \omega  \frac{\partial}{\partial \theta}$$
is a Killing field for this metric, where $\partial/\partial\theta$ is the usual generator of  rotation of $S^{n-1}\subset \RR^n$ around  $\RR^{n-2}\subset \RR^n$. 
Now a flow line of a Killing field is a geodesic iff it passes through a critical point of ${\zap g}(\xi , \xi)$; indeed, Killing's equation $\nabla_{(a}\xi_{b)}=0$ tells us that 
$$\left( \nabla_\xi \xi\right)_b= \xi^a \nabla_a \xi_b= - \xi^a\nabla_b \xi_a =- \frac{1}{2} \nabla_b\xi^a \xi_a =  -\frac{1}{2} \nabla_b \, {\zap g}( \xi , \xi),$$
and the claim therefore follows from the fact that ${\zap g}( \xi , \xi)$ is  constant along the flow. However, restricting to the great circle in 
$S^{n-1}$ where $\partial/\partial \theta$ has maximal length, 
$$
{\zap g}(\xi , \xi ) = -\left(1- \frac{A}{\varrho^{n-2}}\right) + \omega^2 \varrho^2~,
$$
and the critical-point condition then reduces to 
$$0 = \frac{d}{d\varrho} \left[-\left(1- \frac{A}{\varrho^{n-2}}\right) + \omega^2 \varrho^2 \right]= - \frac{(n-2)A}{\varrho^{n-1}}+ 2\omega^2\varrho~. 
$$
Thus, when 
$$
\omega^2= \frac{(n-2)}{2}  \frac{A}{\varrho^n}
$$
a flow-line is a space-time geodesic, and represents a test particle moving in a  circular orbit  at constant angular velocity $\omega$. 
Comparison with circular orbits in the Newtonian model discussed above therefore
 leads us to interpret the Schwarzschild metric as representing the gravitational field of an object of mass 
$$
{\zap M} = \frac{A}{2}  . 
$$

The  spatial slice $t=0$ of the Schwarzschild metric 
$$g= \left( 1- \frac{A}{\varrho^{n-2}}\right)^{-1} d\varrho^2+ \varrho^2 \mathfrak{h}~$$
is totally geodesic, and provides the prototype  for defining the mass of an ALE manifold. 
If we interpret $\varrho$ as the Euclidean radius  in $\RR^n$, this metric takes the form
$$g_{jk} = \delta_{jk} + \frac{A}{\varrho^{n}} x_jx_k + O(\frac{1}{\varrho^{n-1}})$$
so that 
$$
g_{jk,\ell} = \frac{A}{\varrho^{n}}(\delta_{j\ell}x_k+ \delta_{k\ell}x_j-n \frac{x_jx_kx_\ell}{\varrho^2}) + O(\frac{1}{\varrho^{n}})
$$
Thus 
$$
g_{ij,i} -g_{ii,j}= \frac{A}{\varrho^{n}}(\delta_{ji}x_i+ \delta_{ii}x_j-\delta_{ij}x_i- \delta_{ij}x_i) + O(\frac{1}{\varrho^{n}})= (n-1)  \frac{A}{\varrho^{n-1}}\nu_j + O(\frac{1}{\varrho^{n}})
$$
and 
\begin{eqnarray*}
\lim_{\varrho\to \infty} \int_{S_\rho} \left[ g_{ij,i} -g_{ii,j}\right] \mathbf{n}^jd\mathfrak{a}_E&=& (n-1) A \, \mbox{Vol} (S^{n-1}) 
\\&=& 2(n-1) {\zap M}\, \mbox{Vol} (S^{n-1})\\&=& \frac{4(n-1)\pi^{n/2}}{\eg (\frac{n}{2})} {\zap M}
\end{eqnarray*}
Thus, defining the mass of an $n$-dimensional ALE manifold (at a given end) to be 
$$
{\zap m}(M, g) := \lim_{\varrho\to \infty}   \frac{\eg (\frac{n}{2})}{4(n-1)\pi^{n/2}} \int_{S_\rho/\Gamma_j} \left[ g_{ij,i} -g_{ii,j}\right] \mathbf{n}^jd\mathfrak{a}_E
$$
will result in a mass of ${\zap m}= {\zap M}$ for the $t=0$ spatial slice of the Schwarzschild metric. In 
particular, when $n=3$, the normalizing constant simplifies to 
$1/16\pi$, which is the well-established value found throughout the literature. 

\bigskip 
 
\noindent {\bf Acknowledgments:} The authors would like to thank the Isaac Newton Institute, Cambridge, 
for its hospitality during the writing of this article. They would also like to thank 
Claudio Arezzo, 
Ronan Conlan, 
Michael Eastwood,
Gary Gibbons, 
Denny Hill, 
Gustav Holzegel,
John Lott, 
Rafe Mazzeo, 
Bianca Santoro, 
Cristiano Spotti, 
Ioana \c{S}uvaina, and the anonymous referees  
   for many useful comments and suggestions.

\vfill
\noindent
{\sc Department of Mathematics, University of Maryland, College Park, MD 20742-4015, USA}

\bigskip 

\noindent
{\sc Department of Mathematics, State University of New York, Stony Brook, NY 11794-3651, USA}


\begin{thebibliography}{10}

\bibitem{alexeev}
{\sc V.~Alexeev}, {\em Classification of log canonical surface singularities:
  arithmetical proof}, in Flips and Abundance for Algebraic Threefolds,
  Soci\'et\'e Math\'ematique de France, Paris, 1992, pp.~47--58.
\newblock Papers from the Second Summer Seminar on Algebraic Geometry held at
  the University of Utah, Salt Lake City, Utah, August 1991, Ast{\'e}risque No.
  211 (1992).

\bibitem{andgrau}
{\sc A.~Andreotti and H.~Grauert}, {\em Th\'eor\`eme de finitude pour la
  cohomologie des espaces complexes}, Bull. Soc. Math. France, 90 (1962),
  pp.~193--259.

\bibitem{arezzomass}
{\sc C.~Arezzo, R.~Lena, and L.~Mazzieri}, {\em On the resolution of extremal
  and constant scalar curvature {K}aehler orbifolds}, Int. Math. Res. Not.
  IMRN,  (2016).
\newblock First published online December 19, 2015, doi:10.1093/imrn/mv346.

\bibitem{adm}
{\sc R.~Arnowitt, S.~Deser, and C.~W. Misner}, {\em Coordinate invariance and
  energy expressions in general relativity.}, Phys. Rev. (2), 122 (1961),
  pp.~997--1006.

\bibitem{baily}
{\sc W.~L. Baily}, {\em On the imbedding of {$V$}-manifolds in projective
  space}, Amer. J. Math., 79 (1957), pp.~403--430.

\bibitem{bpv}
{\sc W.~Barth, C.~Peters, and A.~Van~de Ven}, {\em Compact Complex Surfaces},
  vol.~4 of Ergebnisse der Mathematik und ihrer Grenzgebiete (3),
  Springer-Verlag, Berlin, 1984.

\bibitem{bartnik}
{\sc R.~Bartnik}, {\em The mass of an asymptotically flat manifold}, Comm. Pure
  Appl. Math., 39 (1986), pp.~661--693.

\bibitem{bes}
{\sc A.~L. Besse}, {\em Einstein Manifolds}, vol.~10 of Ergebnisse der
  Mathematik und ihrer Grenzgebiete (3), Springer-Verlag, Berlin, 1987.

\bibitem{braymass}
{\sc H.~L. Bray}, {\em Proof of the {R}iemannian {P}enrose inequality using the
  positive mass theorem}, J. Differential Geom., 59 (2001), pp.~177--267.

\bibitem{brill}
{\sc D.~R. Brill}, {\em On the positive definite mass of the
  {B}ondi-{W}eber-{W}heeler time-symmetric gravitational waves}, Ann. Physics,
  7 (1959), pp.~466--483.

\bibitem{calabix}
{\sc E.~Calabi}, {\em Extremal {K}\"ahler metrics}, in Seminar on Differential
  Geometry, vol.~102 of Ann. Math. Studies, Princeton Univ. Press, Princeton,
  N.J., 1982, pp.~259--290.

\bibitem{caldsing}
{\sc D.~M.~J. Calderbank and M.~A. Singer}, {\em Einstein metrics and complex
  singularities}, Invent. Math., 156 (2004), pp.~405--443.

\bibitem{carpar}
{\sc G.~Carron}, {\em Cohomologie {$L^2$} et parabolicit\'e}, J. Geom. Anal.,
  15 (2005), pp.~391--404.

\bibitem{chenlebweb}
{\sc X.~X. Chen, C.~LeBrun, and B.~Weber}, {\em On conformally {K}\"ahler,
  {E}instein manifolds}, J. Amer. Math. Soc., 21 (2008), pp.~1137--1168.

\bibitem{admchrus}
{\sc P.~Chru{\'s}ciel}, {\em Boundary conditions at spatial infinity from a
  {H}amiltonian point of view}, in Topological Properties and Global Structure
  of Space-Time ({E}rice, 1985), vol.~138 of NATO Adv. Sci. Inst. Ser. B Phys.,
  Plenum, New York, 1986, pp.~49--59.
\newblock Digitized version available at {\tt \footnotesize
  http://homepage.univie.ac.at/piotr.chrusciel/scans/index.html}.

\bibitem{prinzip}
{\sc M.~Commichau and H.~Grauert}, {\em Das formale {P}rinzip f\"ur kompakte
  komplexe {U}ntermannigfaltigkeiten mit {$1$}-positivem {N}ormalenb\"undel},
  in Recent Developments in Several Complex Variables, vol.~100 of Ann. of
  Math. Stud., Princeton Univ. Press, Princeton, N.J., 1981, pp.~101--126.

\bibitem{heinlon2}
{\sc R.~J. Conlon and H.-J. Hein}, {\em Asymptotically conical {C}alabi-{Y}au
  manifolds, {III}}.
\newblock e-print arXiv:1405.7140 [math.DG], 2014.

\bibitem{thick}
{\sc M.~Eastwood and C.~LeBrun}, {\em Thickening and supersymmetric extensions
  of complex manifolds}, Amer. J. Math., 108 (1986), pp.~1177--1192.

\bibitem{gibhawk}
{\sc G.~W. Gibbons and S.~W. Hawking}, {\em Classification of gravitational
  instanton symmetries}, Comm. Math. Phys., 66 (1979), pp.~291--310.

\bibitem{grauertex}
{\sc H.~Grauert}, {\em \"{U}ber {M}odifikationen und exzeptionelle analytische
  {M}engen}, Math. Ann., 146 (1962), pp.~331--368.

\bibitem{GH}
{\sc P.~Griffiths and J.~Harris}, {\em Principles of Algebraic Geometry},
  Wiley-Interscience, New York, 1978.

\bibitem{hhn}
{\sc M.~Haskins, H.-J. Hein, and J.~Nordstr{\"o}m}, {\em Asymptotically
  cylindrical {C}alabi-{Y}au manifolds}, J. Differential Geom., 101 (2015),
  pp.~213--265.

\bibitem{hahuma}
{\sc T.~Hausel, E.~Hunsicker, and R.~Mazzeo}, {\em Hodge cohomology of
  gravitational instantons}, Duke Math. J., 122 (2004), pp.~485--548.

\bibitem{hiltay}
{\sc C.~D. Hill and M.~Taylor}, {\em Integrability of rough almost complex
  structures}, J. Geom. Anal., 13 (2003), pp.~163--172.

\bibitem{formal}
{\sc A.~Hirschowitz}, {\em On the convergence of formal equivalence between
  embeddings}, Ann. of Math. (2), 113 (1981), pp.~501--514.

\bibitem{hitproj}
{\sc N.~J. Hitchin}, {\em Complex manifolds and {E}instein's equations}, in
  Twistor {G}eometry and {N}onlinear {S}ystems (Primorsko, 1980), Springer,
  1982, pp.~73--99.

\bibitem{hondale2}
{\sc N.~Honda}, {\em Scalar flat {K}\"ahler metrics on affine bundles over
  {$\Bbb C\Bbb P^1$}}, SIGMA. Symmetry Integrability Geom. Methods Appl., 10
  (2014), pp.~1--25, Paper 046.

\bibitem{huilmpen}
{\sc G.~Huisken and T.~Ilmanen}, {\em The {R}iemannian {P}enrose inequality},
  Internat. Math. Res. Notices,  (1997), pp.~1045--1058.

\bibitem{huybr}
{\sc D.~Huybrechts}, {\em Complex Geometry. An Introduction}, Universitext,
  Springer-Verlag, Berlin, 2005.

\bibitem{joycebook}
{\sc D.~D. Joyce}, {\em Compact Manifolds with Special Holonomy}, Oxford
  Mathematical Monographs, Oxford University Press, Oxford, 2000.

\bibitem{kodsub}
{\sc K.~Kodaira}, {\em A theorem of completeness of characteristic systems for
  analytic families of compact submanifolds of complex manifolds}, Ann. of
  Math. (2), 75 (1962), pp.~146--162.

\bibitem{lebthes}
{\sc C.~LeBrun}, {\em Spaces of {C}omplex {G}eodesics and {R}elated
  {S}tructures}, PhD thesis, Oxford University, 1980.
\newblock Digitized version available at {\tt \footnotesize
  http://ora.ox.ac.uk/objects/uuid:e29dd99c-0437-4956-8280-89dda76fa3f8}.

\bibitem{lpa}
\leavevmode\vrule height 2pt depth -1.6pt width 23pt, {\em Counter-examples to
  the generalized positive action conjecture}, Comm. Math. Phys., 118 (1988),
  pp.~591--596.

\bibitem{mcp2}
\leavevmode\vrule height 2pt depth -1.6pt width 23pt, {\em Explicit self-dual
  metrics on {${\mathbb {C}}{\mathbb {P}}\sb 2\#\cdots\#{\mathbb {C}}{\mathbb
  {P}}\sb 2$}}, J. Differential Geom., 34 (1991), pp.~223--253.

\bibitem{lebmero}
\leavevmode\vrule height 2pt depth -1.6pt width 23pt, {\em Twistors, {K}\"ahler
  manifolds, and bimeromorphic geometry. {I}}, J. Amer. Math. Soc., 5 (1992),
  pp.~289--316.

\bibitem{lebmask}
{\sc C.~LeBrun and B.~Maskit}, {\em On optimal 4-dimensional metrics}, J. Geom.
  Anal., 18 (2008), pp.~537--564.

\bibitem{lebpoonsym}
{\sc C.~LeBrun and Y.~S. Poon}, {\em Self-dual manifolds with symmetry}, in
  Differential geometry: geometry in mathematical physics and related topics
  (Los Angeles, CA, 1990), vol.~54 of Proc. Sympos. Pure Math., Amer. Math.
  Soc., Providence, RI, 1993, pp.~365--377.

\bibitem{lp}
{\sc J.~Lee and T.~Parker}, {\em The {Y}amabe problem}, Bull. Am. Math. Soc.,
  17 (1987), pp.~37--91.

\bibitem{vialock}
{\sc M.~Lock and J.~Viaclovsky}, {\em A {sm\"{o}rg{\aa}sbord} of scalar-flat
  {K\"{a}hler} {ALE} surfaces}.
\newblock e-print arXiv:1410.6461 [math.DG], 2014.

\bibitem{malgrange}
{\sc B.~Malgrange}, {\em Sur l'int\'egrabilit\'e des structures
  presque-complexes}, in Symposia Mathematica, Vol. II (INDAM, Rome, 1968),
  Academic Press, London, 1969, pp.~289--296.

\bibitem{marshthes}
{\sc S.~Marshall}, {\em Deformations of {S}pecial {L}agrangian {S}ubmanifolds},
  PhD thesis, Oxford University, 2002.
\newblock Digitized version available at {\tt \footnotesize
  http://people.maths.ox.ac.uk/joyce/theses/theses.html}.

\bibitem{nijenwoo}
{\sc A.~Nijenhuis and W.~B. Woolf}, {\em Some integration problems in
  almost-complex and complex manifolds.}, Ann. of Math. (2), 77 (1963),
  pp.~424--489.

\bibitem{penineq}
{\sc R.~Penrose}, {\em Naked singularities}, Ann. New York Acad. Sci., 224
  (1973), pp.~125--134.
\newblock Sixth Texas Symposium on Relativistic Astrophysics.

\bibitem{rosimass}
{\sc Y.~Rollin and M.~Singer}, {\em Constant scalar curvature {K}\"ahler
  surfaces and parabolic polystability}, J. Geom. Anal., 19 (2009),
  pp.~107--136.

\bibitem{rossivec}
{\sc H.~Rossi}, {\em Vector fields on analytic spaces}, Ann. of Math. (2), 78
  (1963), pp.~455--467.

\bibitem{syaction}
{\sc R.~M. Schoen and S.~T. Yau}, {\em Complete manifolds with nonnegative
  scalar curvature and the positive action conjecture in general relativity},
  Proc. Nat. Acad. Sci. U.S.A., 76 (1979), pp.~1024--1025.

\bibitem{schoustr}
{\sc J.~{Schouten} and D.~{Struik}}, {\em {Einf\"uhrung in die neueren Methoden
  der Differentialgeometrie, vol. 2}}, {P. Noordhoff}, {Groningen}, 1938.

\bibitem{schouten}
{\sc J.~A. Schouten}, {\em Ricci-{C}alculus. {A}n Introduction to Tensor
  Analysis and its Geometrical Applications}, Springer-Verlag, Berlin, 1954.
\newblock 2nd ed.

\bibitem{weber}
{\sc B.~Weber}, {\em First {Betti} numbers of {K\"{a}hler} manifolds with
  weakly pseudo-convex boundary}.
\newblock e-print arXiv:1110.4571 [math.DG], 2011.

\bibitem{weylproj}
{\sc H.~{Weyl}}, {\em {Zur Infinitesimalgeometrie: Einordnung der projektiven
  und der konformen Auffassung}}, {Nachr. Ges. Wiss. G\"ottingen, Math.-Phys.
  Kl.}, 1921 (1921), pp.~99--112.

\bibitem{witmass}
{\sc E.~Witten}, {\em A new proof of the positive energy theorem}, Comm. Math.
  Phys., 80 (1981), pp.~381--402.

\end{thebibliography}
\end{document}